[1994/12/01]

\documentclass{amsart}
\usepackage{graphicx}
 \usepackage{stmaryrd}
 \usepackage{amssymb}
\usepackage[dvipsnames]{xcolor}
\usepackage{enumerate}
\usepackage{amsrefs}

\usepackage[colorlinks=true,
            linkcolor=violet,
            urlcolor=Aquamarine,
            citecolor=YellowOrange]{hyperref}
\definecolor{darkgreen}{rgb}{0,0.5,0}
\definecolor{darkred}{rgb}{0.7,0,0}

\def\a{\alpha}
\def\b{\beta}
\def\({\left (}
\def\){\right )}
\def\<{\left\langle}
\def\>{\right\rangle}

\newtheorem{thm}{Theorem}[section]
\newtheorem{cor}[thm]{Corollary}
\newtheorem{lem}[thm]{Lemma}

\newtheorem{defn}[thm]{Definition}
\newtheorem{rem}[thm]{Remark}
\newtheorem*{rem*}{Remark}

\numberwithin{equation}{section}

\newcommand{\norm}[1]{\left\Vert#1\right\Vert}
\newcommand{\abs}[1]{\left\vert#1\right\vert}
\newcommand{\set}[1]{\left\{#1\right\}}
\newcommand{\Real}{\mathbb R}

\newcommand{\pfrac}[2]{\frac{\partial #1}{\partial #2}}

\newcommand{\osc}{\mbox{osc}}

\title[Expansion formula]
{Expansion formula for complex Monge-Amp\`ere equation along cone singularities}
\author{Hao Yin}
  \address{School of mathematical sciences, university of science and technology
of China, Hefei, 230026, China}
  \email[Hao Yin]{haoyin@ustc.edu.cn}

\author{Kai Zheng}
  \address{Mathematics Institute, University of Warwick, Coventry, CV4 7AL, UK}
  \email[Kai Zheng]{K.Zheng@warwick.ac.uk}

\date{\today}

\begin{document}

\maketitle
\begin{abstract}
In this paper, we prove the asymptotic expansion of the solutions to some singular complex Monge-Amp\`ere equation which arise naturally in the study of the conical K\"ahler-Einstein metric.	
\end{abstract}

\section{Introduction}\label{sec:intro}

Let $M$ be a closed K\"ahler manifold and $D$ be a smooth divisor in $M$. A K\"ahler metric $g$ is said to be of cone angle $2\pi \beta$ (with $0<\beta<1$) along $D$ if $g$ is smooth away from $D$ and for each $p\in D$, there is a holomorphic coordinates $\set{z_1,\cdots,z_n}$ around $p$ with $D=\set{z_1=0}$ such that $g$ is comparable to the standard cone metric
\begin{equation*}
	g_{cone}= \abs{z}_1^{2\beta-2} dz_1d\bar{z}_1 + dz_2d\bar{z}_2+ \cdots + dz_n d\bar{z}_n.
\end{equation*}
If the metric is Einstein away from the divisor, then it is called conical K\"ahler-Einstein metric. 

The study of K\"ahler cone metrics could be traced back to Tian \cite{tian1996kahler}. Since then, there are researches on uniqueness and existence of conical K\"ahler-Einstein metrics \cite{jeffres2000uniqueness, mazzeo1999kahler}. Recently, in \cite{donaldson2012kahler}, Donaldson introduced a new function space $C^{2,\a}_\b$ (see Section \ref{sec:setup} for the definition) and proved a Schauder estimate of the linear equation, which provides the key analysis tool and stimulates more research work along this line. 
It is now known that the uniqueness and existence of conical K\"ahler-Einstein metrics in $C^{2,\alpha}_\beta$ space, i.e. the K\"ahler potential of the conical K\"ahler-Einstein metric lies in this new space, could be proved without knowing the higher order regularity near the divisors (see \cite{li2015continuity} and references therein). Nevertheless, the expansion formula has been used in \cite{MR3470713} to prove a Chern number inequality for some conical K\"ahler metrics. Therefore, it is interesting to understand it down to earth and more accurately, and this is the main topic we would explore in this paper.

For those metrics generated by using Donaldson's Schauder estimate, the K\"ahler potential (together with its tangential derivatives) lies automatically in the $C^{2,\alpha}_\beta$ space.  Under the angle restriction $0<\beta<\frac{1}{2}$, the higher regularity near the divisor is studied by Brendle \cite{brendle2013ricci}, meanwhile the higher order Donaldson's spaces are defined and used to improve the regularity of the more general constant scalar curvature K\"ahler metric with cone singularities in \cite{calamai2015geodesics,li2016uniqueness}. For $0<\beta<1$, Jeffres, Mazzeo and Rubinstein \cite{jeffres2011k} proved that the metric is polyhomogeneous, in the sense that
\begin{equation*}
	\varphi_{KE}(r,\theta,Z)\sim \sum_{j,k\geq 0}\sum_{l=0}^{N_{j,k}} a_{j,k,l}(\theta,Z) r^{j+k/\beta} (\log r)^l
\end{equation*}
where $r=\abs{z_1}^\beta/\beta$, $\theta=\mbox{arg} z_1$ and $Z=(z_2,\cdots,z_n)$. We refer the readers to Theorem 2 of \cite{jeffres2011k} for the complete detail of this result.

The main result of this paper is to show that by following a method of \cite{yin2016}, which is very different from the method in \cite{jeffres2011k}, we can prove an expansion with more detailed information. Among other things, we can show that 
\begin{equation*}
	N_{j,k}\leq \max \set{0,k-1}\quad \text{and}\quad   j \text{ is even.}
\end{equation*}
 Briefly speaking, the extra information is obtained by making full use of the fact that the section of the cone (modulo tangential directions) is $S^1$. Moreover, we shall see in the proof of the main theorem that the exact form of the expansion relies both on this geometric structure of the singularity and on the particular nonlinear structure of the complex Monge-Amp\`ere equation. It makes a very interesting comparison to note that in \cite{yin2016}, the author showed that the expansion of the Ricci flow solution on conical surfaces involves no $\log$ term.

Since the nature of the regularity problem is local, we start with a local form. Its relation to the conical K\"ahler-Einstein metric shall be clear in a minute. Consider a solution $\varphi$ to the following singular Monge-Amp\`ere equation
\begin{equation}\label{eqn:main}
	\det (\varphi_{i\bar{j}}) = \frac{e^{\lambda \varphi +h}}{\abs{z_1}^{2-2\beta}}
\end{equation}
on $B_1\subset \mathbb C^n$. Here $B_1$ is the ball of radius $1$ centered at the origin. Assume that  
\begin{enumerate}[(S1)]
	\item $h$ is a smooth function of $z=(z_1,\dots,z_n)\in B_1\subset \mathbb C^n$; 
	\item $\varphi$ is in $C_\beta^{2,\alpha}$ space (see Section \ref{subsec:estimate} for the definition);
	\item there is some constant $c>1$ such that
		\begin{equation*}
			\frac{1}{c}\omega_{cone}\leq \sqrt{-1}\partial\bar{\partial} \varphi \leq c\omega_{cone},
		\end{equation*}
		where $\omega_{cone}=\sqrt{-1} \partial \bar{\partial} \left( \frac{1}{\beta^2}\abs{z_1}^{2\beta} + \sum_{i=2}^n \abs{z_i}^2 \right)$.
\end{enumerate}

To describe the regularity of the solution $\varphi$, we need some definitions. Define the polar coordinates $(\rho,\theta,\xi)$ by
		\begin{equation*}
			\rho=\frac{1}{\beta}\abs{z_1}^\beta, \quad z_1= \abs{z_1}e^{i\theta}, \quad \xi=(z_2,\cdots,z_n).
		\end{equation*} 
		Let $\mathcal T_{log}$ be the set of functions
\begin{equation*}
	\rho^{2j+\frac{k}{\beta}} (\log \rho)^m \cos l\theta,\rho^{2j+\frac{k}{\beta}} (\log \rho)^m \sin l\theta
\end{equation*}
with $j,k,m,l$ satisfying
\begin{enumerate}[(T1)]
	\item $k,j,l,m=0,1,2,\cdots$;
	\item $\frac{k-l}{2}\in \mathbb N\cup\set{0}$;
	\item $m\leq \max\set{0,k-1}$.
\end{enumerate}

The main theorem of this paper is
\begin{thm}
	\label{thm:main}
	Suppose that $\varphi$ is a solution to \eqref{eqn:main} on $B_1\subset \mathbb C^n$ and that (S1)-(S3) hold. Then for each fixed $\xi$ with $\abs{\xi}<\frac{1}{2}$ and $q>0$, $\varphi$ has an expansion up to order $q$ in the sense that there exists $\eta$ in $Span(\mathcal T_{log})$ such that
	\begin{equation*}
		\varphi(\rho,\theta,\xi)= \eta + \tilde{O}(q).
	\end{equation*}
	Here $Span(\mathcal T_{log})$ is the vector space of finite linear combinations of $\mathcal T_{log}$ and $\tilde{O}(q)$ stands for a function of $\rho,\theta$ satisfying
	\begin{equation*}
		\abs{(\rho\partial_\rho)^{k_1} \partial_\theta^{k_2}\tilde{O}(q)}\leq C(k_1,k_2,\xi) \rho^q \quad \text{for any} \quad k_1,k_2\in \mathbb N\cup \set{0}.
	\end{equation*}
	Moreover, any derivatives of $\varphi$ along $\xi$ direction have expansion up to any order in the same sense.
\end{thm}

Before we discuss the application of Theorem \ref{thm:main}, we briefly introduce the ideas involved in its proof. 

It follows from Donaldson's linear theory that the tangential regularity of $\varphi$ is not a problem at all, hence we treat $z_2, \cdots, z_n$ as parameters and study a regularity problem on conical surface. For that purpose, we set $\tilde{\triangle}= \partial_\rho^2 +\frac{1}{\rho}\partial_\rho + \frac{1}{\beta^2 \rho^2} \partial_\theta^2$ and rewrite \eqref{eqn:main} into the form
\begin{equation*}
	\tilde{\triangle} \varphi = \text{RHS}
\end{equation*}
where the $\text{RHS}$ is an expression involving derivatives of $\varphi$(see Section \ref{subsec:equation}). 

The first ingredient of the proof is a formal analysis, which we use to decide which functions are needed for the expansion of $\varphi$. Roughly speaking, we pretend that $\varphi$ and its tangential derivatives {\it are} finite linear combinations of functions in $\mathcal T$ for some set of functions $\mathcal T$ and compute the $\text{RHS}$ above so that the result is the linear combination of some other set of functions $\mathcal T_{rhs}$, which may not be identical to $\mathcal T$. We look for the smallest $Span(\mathcal T)$ such that every function in $\mathcal T_{rhs}$ lies in the image $\tilde{\triangle}(Span(\mathcal T))$. Moreover, we require that $Span(\mathcal T)$ contains the terms in the expansion of bounded harmonic functions (see (1) in Lemma \ref{lem:harmonic}). This is how we obtain $\mathcal T_{\log}$.

The second ingredient is an estimate of $\varphi$ (Theorem \ref{thm:interior}) away from the singular set $\set{z_1=0}$, which serves as the starting point of the bootstrapping argument in the proof of Theorem \ref{thm:main}. The idea is that away from the singular set, since the complex Monge-Amp\`ere equation is elliptic and that we have assumed Donaldson's $C^{2,\alpha}_\beta$ estimate, we should be able to get estimates for higher order derivatives of $\varphi$, which blow up at some fixed rate near the singular set. 

The rest of the proof is a bootstrapping argument. Once we know that $\varphi$ and its tangential derivatives have expansion up to a certain order, we can improve this order by at least $1$. The details of this argument appear in Section \ref{subsec:finite} and \ref{subsec:proof}.

\vskip 2mm

To apply Theorem \ref{thm:main} to the regularity problem of conical K\"ahler-Einstein metric, we briefly recall the basic setting. Suppose that $M$ is a compact K\"ahler manifold with a smooth K\"ahler form $\omega_0$ and $D$ is a smooth divisor in $M$, whose corresponding line bundle $L$ has a global holomorphic section $s$ vanishing on $D$. Assume that 
\begin{equation}\label{eqn:homology}
	c_1(M)= \mu [\omega_0] + (1-\beta) [D]
\end{equation}
for some $\mu\in \Real$. Here $[D]$ is the cohomology class defined by the closed $(1,1)$ current defined by the divisor $D$.  

By \eqref{eqn:homology}, there exists an (smooth) hermitian metric $h_0$ of $L$ such that its curvature form $\Theta_{h_0}$ satisfies
\begin{equation}\label{eqn:R2}
	Ric(\omega_0) = \mu \omega_0 + (1-\beta) \Theta_{h_0}
\end{equation}

For $\delta$ sufficiently small,
\begin{equation*}
	\omega_D= \omega_0 + \frac{\delta i}{2\pi}\partial\bar{\partial} \abs{s}_{h_0}^{2\beta}
\end{equation*}
is a K\"ahler metric on $M\setminus D$ and is asymptotically a cone metric along $D$ (see Section 4.3 in \cite{donaldson2012kahler}). For $\psi\in C^{2,\alpha}_\beta(M)$ (the Donaldson H\"older space, see Section 2 for the precise definition), the K\"ahler metric
\begin{equation*}
	\omega_\psi=\omega_D + \frac{i}{2\pi} \partial \bar{\partial} \psi
\end{equation*}
is called a conical K\"ahler-Einstein metric if 
\begin{equation*}
	Ric(\omega_\psi) = \mu \omega_\psi + (1-\beta) D
\end{equation*}
in the sense of currents, where by abuse of notation, we also use $D$ for the current associated to the divisor $D$. By the Poincar\'e-Lelong formula, this is equivalent to
\begin{equation}\label{eqn:R1}
	Ric(\omega_\psi) = \mu \omega_\psi + \frac{i}{2\pi}\partial\bar{\partial} \abs{s}_{h_0}^{2(1-\beta)} + (1-\beta) \Theta_{h_0}.
\end{equation}

Subtracting \eqref{eqn:R2} from \eqref{eqn:R1} yields (in the sense of currents)
\begin{equation*}
	\partial\bar{\partial} \log \frac{\omega_\psi^n}{\omega_0^n} = \partial \bar{\partial} (\mu \psi + \mu \delta \abs{s}_{h_0}^{2\beta}) + \partial \bar{\partial} \abs{s}_{h_0}^{2(1-\beta)},
\end{equation*}
which implies
\begin{equation}
	\frac{\omega_\psi^n}{\omega_0^n}= \abs{s}_{h_0}^{2(1-\beta)} e^{\mu (\psi + \delta \abs{s}_{h_0}^{2\beta})}.
	\label{eqn:Rglobal}
\end{equation}

To reduce the global equation \eqref{eqn:Rglobal} into a local one, we take a holomorphic coordinate system $\set{z_i}$ around some $p\in D$ such that $D$ (in this neighborhood) is given by $\set{z_1=0}$. We pick the trivialization of $L$ such that $s$ is given by $z_1$ and denote the hermitian metric in this trivialization by a real-valued function $\tilde{h}$. Moreover, we assume that $\omega_0= \frac{i}{2\pi} \partial \bar{\partial} \psi_0$. Keeping the above notations in mind and setting
\begin{equation}\label{eqn:varphi}
	\varphi = \psi_0 + \delta \abs{s}_{h_0}^{2\beta} + \psi,
\end{equation}
we obtain
\begin{equation}\label{eqn:cor}
	\det \varphi_{i\bar{j}} = \det (\psi_0)_{i\bar{j}}  \abs{z_1}^{2(2-\beta)} \tilde{h}^{1-\beta} e^{ \mu (\varphi- \psi_0)},
\end{equation}
which is just equation \eqref{eqn:main} if we take
$\lambda=\mu$ and
\begin{equation}\label{eqn:mainh}
	h= \log \left( \det (\psi_0)_{i\bar{j}} \cdot \tilde{h}^{1-\beta} e^{ -\mu \psi_0} \right).
\end{equation}
Hence, we can apply Theorem \ref{thm:main} to get
\begin{cor}
	\label{cor:main} If $\omega_\psi$ is a conical K\"ahler-Einstein metric as defined above, then $\varphi$ (hence $\psi$) has an expansion up to any order as defined in Theorem \ref{thm:main}.
\end{cor}
The proof of this corollary is given in Section \ref{subsec:cor}.

\bigskip

\textbf{Acknowledgements}. 
The work of H. Yin is supported by NSFC 11471300.
The work of K. Zheng has received funding from the European Union's Horizon 2020 research and innovation programme under the Marie Sk{\l}odowska-Curie grant agreement No 703949, and was also partially supported by the Engineering and Physical Sciences Research Council (EPSRC) on a Programme Grant entitled "Singularities of Geometric Partial Differential Equations" reference number EP/K00865X/1.

\section{Preliminaries}\label{sec:setup}
We include in this section a few elementary discussions needed for this paper. First, we define several coordinate systems and explain the notations that we use. Second, we prove a few lemmas on how estimates are translated between different coordinate systems. Finally, we rewrite the complex Monge-Amp\`ere equation for the use of the proof of our main result.  

\subsection{Various coordinates}\label{subsec:coordinate}
Recall that $D$ is a smooth divisor of the K\"ahler manifold $M$. We shall list a few coordinate systems which appear naturally in the study of conical K\"ahler metric. Please note that some of them are defined in a neighborhood of any point in $D$ and some are defined in a small ball away from $D$. 

\begin{enumerate}[(1)]
	\item The {\bf holomorphic coordinates} $z$. This is the holomorphic coordinates $(z_1,\cdots,z_n)$ of the underlying complex manifold such that the divisor $D$ is given by $z_1=0$. It is the most natural coordinate system and we have introduced our equation \eqref{eqn:main} using it. Partial derivatives of $\varphi$ with respect to $z$ coordinates are denoted by $\varphi_i, \varphi_{\bar{j}}$, $\varphi_{i\bar{j}}$ and so on. The roles played by $z_1$ and $z_l(l=2,\cdots,n)$ are very different. To emphasize this, we usually write $\xi$ for any one of $z_2,\cdots,z_n$. Hence, $\varphi_\xi$ can be any one of $\varphi_2,\cdots,\varphi_n$ and the same convention applies to $\varphi_{\bar{\xi}}$, $\varphi_{\xi\bar{\xi}}$ and so on.

	\item The {\bf polar coordinates} $(\rho,\theta,\xi)$. If $(z_1,\cdots,z_n)$ is the holomorphic coordinates above, the polar coordinates are defined by
		\begin{equation*}
			\rho=\frac{1}{\beta}\abs{z_1}^\beta, \quad z_1= \abs{z_1}e^{i\theta}, \quad \xi=(z_2,\cdots,z_n).
		\end{equation*} 
		Here we take $\xi$ as a real vector in $\Real^{2n-2}$.
		Partial derivatives are denoted by $\partial_\rho$, $\partial_\theta$, $\nabla_\xi$, $\nabla^2_\xi$ and so on.

	\item The {\bf lifted holomorphic coordinates} $v$ (or $V$-coordinate for short). This coordinate system is only defined in small neighborhood away from (but near) $D$. More precisely, given a point $Z_0=(\rho_0,\theta_0,\xi_0)$ (in the polar coordinates) with $\rho_0\ne 0$ and $\theta_0\in [0,2\pi)$, we consider a neighborhood of $Z_0$ defined by
\begin{equation}\label{eqn:omega}
	\Omega = \set{(\rho,\theta,\xi)|\quad \rho_0/2<\rho<2\rho_0, \theta_0-0.1<\theta<\theta_0+0.1, \abs{\xi_0-\xi}<\rho_0}.
\end{equation}
For points in $\Omega$, it makes sense to define
\begin{equation*}
	v_1:=\frac{1}{\beta}z_1^\beta=\frac{1}{\beta} \abs{z_1}^\beta e^{i \beta \theta}.
\end{equation*}
We write $v=(v_1,\xi)=(v_1,\cdots,v_n)\in \mathbb C^n$. Note that for any such $Z_0$, the lifted holomorphic coordinates are defined for $\abs{v-v_0}\leq c_\beta \rho_0$ for some constant $c_\beta<1$ (depending only on $\beta$) and $v_0=(\rho_0 e^{i\beta\theta_0}, \xi_0)$. Partial derivatives are denoted by $\varphi_{V,i}, \varphi_{V,\bar{j}}, \varphi_{V, i\bar{j}}$ and so on. As before, we use $\nabla_{V,\xi}$ to indicate $\partial_{v_i}$ or $\partial_{\bar{v}_i}$ for $i=2,\cdots,n$.

\item The {\bf scaled lifted holomorphic coordinates} $\tilde{v}$ (or $\tilde{V}$-coordinate for short). This is simply a scaling (by $\rho_0$) and translation of $v$ (by $v_0$) so that $\tilde{v}$ is defined in $B_{c_\beta}(0)\subset \mathbb C^n$. Precisely,
	\begin{equation*}
		\tilde{v}=\frac{v-v_0}{\rho_0}.
	\end{equation*}
	Partial derivatives are denoted by $\varphi_{\tilde{V},i}, \varphi_{\tilde{V},\bar{j}}$, $\varphi_{\tilde{V},i\bar{j}}$ and so on and $\nabla_{\tilde{V},\xi}$ is understood as in $V$-coordinates.
\end{enumerate}

We also define H\"older spaces by using (scaled) lifted holomorphic coordinates. 
$C^{k,\alpha}_{V}$ is the set of function $f:B_{c_\beta \rho_0}(v_0)\to \Real$ that is $C^{k,\alpha}$ in the usual sense (i.e. with respect to the derivative and distance given by the $v$-coordinates) and the norm $\norm{\cdot}_{C^{k,\alpha}_V}$ is also the same with the usual H\"older norm $\norm{\cdot}_{C^{k,\alpha}(B_{c_\beta \rho_0}(v_0))}$. Here we use the subscript to emphasize that we are using the $V$-coordinates around $v_0$. Similar convention holds for $C^{k,\alpha}_{\tilde{V}}$.


\subsection{Estimates in various coordinates}\label{subsec:estimate}

There is little doubt that all the above mentioned coordinate systems are important in the study of conical K\"ahler metrics and that with some efforts, one can switch between them if necessary. In this section, we prove a few lemmas along this line. 

First, we show how a bound of weighted derivatives (in polar coordinates) can be proved by using the lifted holomorphic coordinates.  

\begin{lem}\label{lem:tildeW}
	Suppose that $u:B_1\subset \mathbb C^n\to \Real$ is a function of holomorphic coordinates. The following are almost equivalent in the sense that: (ii) implies (i), while (i) implies (ii) if the $1/2$'s in (ii) are replaced by $1/4$.
	
	(i) For any $k_1,k_2,k_3\in \mathbb Z^+\cup\set{0}$, there exists constant $C(k_1,k_2,k_3)>0$ such that
	\begin{equation*}
		\abs{(\rho \partial_\rho)^{k_1} \partial_\theta^{k_2} (\rho \nabla_{\xi})^{k_3} u} \leq C(k_1,k_2,k_3)
	\end{equation*}
	for any $\rho\in (0,1/2)$ and any $\abs{\xi}<1/2$.

	(ii) For any $k\in \mathbb Z^+\cup\set{0}$, any $Z_0= (\rho_0,\theta_0,\xi_0)$ with $\rho_0\in (0,1/2)$ and $\abs{\xi_0}<1/2$, the $k$-th order derivatives of $u$ on $B_{c_\beta}$ in scaled lifted holomorphic coordinates are bounded by a constant $C(k)$ (independent of $Z_0$).
\end{lem}
\begin{proof}
	The proof is based on the following computation.
	For the $\rho$ and $\theta$ part, recall that $v_1=\frac{1}{\beta}z_1^\beta$ and $z_1=r e^{i\theta}$, which implies
	\begin{equation*}
	\partial_{v_1}= z_1^{1-\beta} \partial_{z_1} \quad \mbox{and}\quad \partial_{\bar v_1}= \bar z_1^{1-\beta} \partial_{\bar z_1}
	\end{equation*}
	and
	\begin{eqnarray*}
		\partial_{z_1}&=&  \frac{1}{2}\left( e^{-i\theta}\partial_r + e^{i(\frac{3}{2}\pi-\theta)}\frac{1}{r}\partial_\theta \right) \\
		\partial_{\bar{z}_1}&=&  \frac{1}{2}\left( e^{i\theta}\partial_r + e^{i(-\frac{3}{2}\pi+\theta)}\frac{1}{r}\partial_\theta \right).
	\end{eqnarray*}
	Putting the above equations together and recalling that $\rho=\frac{1}{\beta}r^\beta$ give 
	\begin{equation}\label{eqn:a1}
		\partial_{\tilde{v}_1}=\frac{\rho_0}{2\rho}\left( e^{-i\beta \theta} (\rho\partial_\rho) + \beta^{-1} e^{i(\frac{3}{2}\pi -\beta \theta)}\partial_\theta \right)	
	\end{equation}
	and
	\begin{equation}\label{eqn:a2}
		\partial_{\bar{\tilde{v}}_1}=\frac{\rho_0}{2\rho}\left( e^{i\beta \theta} (\rho\partial_\rho) + \beta^{-1} e^{i(-\frac{3}{2}\pi +\beta \theta)}\partial_\theta \right).
	\end{equation}
	Along the tangent direction, for $l=2,\cdots,n$,
	\begin{equation}\label{eqn:a3}
		\begin{split}
			\partial_{\tilde{v}_l} &= \rho_0 \partial_{z_l} = \frac{\rho_0}{2\rho} (\rho \partial_{x_l} - i \rho\partial_{y_l}) \\
			\partial_{\bar{\tilde{v}}_l} &= \rho_0 \partial_{\bar{z}_l} = \frac{\rho_0}{2\rho} (\rho \partial_{x_l} + i \rho \partial_{y_l}).
		\end{split}
	\end{equation}
	Here $x_l$ and $y_l$ are the real and imaginary part of $z_l$ so that $\nabla_\xi= (\partial_{x_1},\ldots,\partial_{x_n},\partial_{y_1},\cdots,\partial_{y_n})$. 
	
	If (i) holds, an immediate consequence of \eqref{eqn:a1}, \eqref{eqn:a2} and \eqref{eqn:a3} is that (ii) holds for $k=1$. For larger $k$, we observe that any coefficient $w$ in front of $\rho \partial_\rho$, $\partial_\theta$ and $\rho \nabla_{\xi}$ in \eqref{eqn:a1},\eqref{eqn:a2} and \eqref{eqn:a3} satisfies
	\begin{equation*}
		\abs{(\rho\partial_\rho)^{l_1} (\partial_\theta)^{l_2} (\rho\nabla_\xi)^{l_3} w} \leq C(l_1,l_2,l_3) \qquad \text{when} \quad \rho\in (\rho_0/2,2\rho_0)
	\end{equation*}
	for some constants $C(l_1,l_2,l_3)$. One may derive from \eqref{eqn:a1}, \eqref{eqn:a2} and \eqref{eqn:a3} the formulas for higher $\tilde{v}$-derivatives and with the help of the above observation, we conclude that (i) implies (ii) for any $k$.

	To see that (ii) implies (i), we solve from \eqref{eqn:a1}, \eqref{eqn:a2} and \eqref{eqn:a3}
	\begin{equation}
		\rho\partial\rho = \frac{\rho}{\rho_0} \left( e^{i \beta \theta} \partial_{\tilde{v}_1} + e^{ -i \beta \theta} \partial_{\bar{\tilde{v}}_1} \right)
		\label{eqn:b1}
	\end{equation}
	\begin{equation}
		\partial_\theta = \frac{\beta \rho i}{\rho_0} \left( e^{i \beta \theta} \partial_{\tilde{v}_1} - e^{ -i \beta \theta} \partial_{\bar{\tilde{v}}_1} \right)
		\label{eqn:b2}
	\end{equation}
	and for $l=2,\cdots,n$,
	\begin{equation}
		\begin{split}
			\rho\partial_{x_l}&= \frac{\rho}{\rho_0} (\partial_{\tilde{v}_l}+\partial_{\bar{\tilde{v}}_l})\\
		\rho\partial_{y_l}&= \frac{\rho}{\rho_0} (\partial_{\tilde{v}_l}- \partial_{\bar{\tilde{v}}_l}).
		\end{split}
		\label{eqn:b3}
	\end{equation}
	If $v$ is any coefficient in front of $\partial_{\tilde{v}_l}$ or $\partial_{\bar{\tilde{v}}_l}$ with $l=1,\cdots,n$ in \eqref{eqn:b1}, \eqref{eqn:b2} and \eqref{eqn:b3}, then $v$ obviously satisfies (i), hence by what we have just proved, $v$ also satisfies (ii). In particular, any $\tilde{v}$-derivatives of $v$ in this same $\tilde{V}$-coordinates around $Z_0$ are bounded. With this observation, (ii) implies (i) by iterated use of \eqref{eqn:b1}, \eqref{eqn:b2} and \eqref{eqn:b3}.
\end{proof}

\begin{cor}\label{cor:h}
	Suppose that $h:B_1\to \Real$ is a smooth function in holomorphic coordinates $z$. Then $h$ satisfies (i) and (ii) in Lemma \ref{lem:tildeW}. In particular, for each $Z_0=(\rho_0,\theta_0,\xi_0)$ with $\rho_0\in (0,1/4)$ and $\abs{\xi_0}<1/4$, if by using the $\tilde{V}$-coordinate around $Z_0$ we take $h$ as a function defined on $B_{c_\beta}$, then
	\begin{equation*}
		\norm{h}_{C_{\tilde{V}}^k(B_{c_\beta})}\leq C(k) \qquad \text{independent of} \quad Z_0.
	\end{equation*}
\end{cor}
\begin{proof}
	Being a smooth function in $z$ implies that
	\begin{equation*}
		\abs{(r\partial_r)^{k_1} (\partial_\theta)^{k_2} (\nabla_\xi)^{k_3}  h}\leq C(k_1,k_2,k_3).
	\end{equation*}
	The corollary follows from Lemma \ref{lem:tildeW} and the fact that $r\partial_r = \beta \rho\partial_\rho$.
\end{proof}

Next, we recall the Donaldson's $C_\beta^{2,\alpha}$ norm and show some implications (in the lifted holomorphic coordinates) when it is bounded. 

With a background conical metric (for example, $g_{cone}$ as given in the introduction), one can define $C^\alpha_\beta(B)$ space and $C_\beta^\alpha(B)$ norm for functions defined on $B$, where $B$ stands for some ball in $\mathbb C^n$ centered at the origin. Moreover, there is a subspace $C^{\alpha}_{\beta,0}(B)$ of $C_\beta^\alpha(B)$ which consists of all $u\in C_\beta^\alpha(B)$ satisfying $u|_{\set{z_1=0}} \equiv 0$. A function $u$ defined on $B$ is said to be in (Donaldson's) $C_\beta^{2,\alpha}$ space if and only if (see Section 2 of \cite{brendle2013ricci})
\begin{enumerate}[(D1)]
	\item $u\in C_\beta^\alpha$
	\item $\abs{z_1}^{1-\beta} \partial_{z_1} u, \abs{z_1}^{1-\beta} \partial_{\bar{z}_1}u \in C^\alpha_{\beta,0}$;
	\item for $k\ne 1$, $\abs{z_1}^{1-\beta} \partial^2_{z_k \bar{z}_1} u$ and $\abs{z_1}^{1-\beta} \partial^2_{z_1\bar{z}_k}u \in C^\alpha_{\beta,0}$;
	\item for $k\ne 1$ and $l\ne 1$, $\abs{z_1}^{2-2\beta} \partial^2_{z_1\bar{z}_1} u$, $\partial^2_{z_k \bar{z}_l} u$ and $\partial^2_{z_k z_l} u\in C_\beta^\alpha$.
\end{enumerate}
Moreover, $\norm{u}_{C^{2,\alpha}_\beta(B)}$ is defined to be the sum of the $C_\beta^\alpha$ norms of all the above functions.

There is an equivalent way (see \cite{brendle2013ricci}) of defining the same space in terms of the polar coordinates: (if we replace (D1-D4) above by) 
\begin{enumerate}[(P1)]
	\item $u\in C_\beta^\alpha$;
	\item $\partial_\rho u, \frac{1}{\rho}\partial_\theta u$ are in $C_{\beta,0}^\alpha$;
	\item $\nabla^2_\xi u, \partial_\rho \nabla_\xi u, \rho^{-1} \partial_\theta \nabla_\xi u$ and $\tilde{{\triangle}} u$ are in $C_\beta^\alpha$
where $\tilde{\triangle}= \partial^2_\rho + \frac{1}{\rho} \partial_\rho + \frac{1}{\beta^2 \rho^2} \partial_\theta^2 $. 
\end{enumerate}

Here is an estimate in the lifted holomorphic coordinates for functions with bounded $C_\beta^{2,\alpha}$ norm. 
\begin{lem} \label{lem:metricinV}
	Suppose that $u\in C_\beta^{2,\alpha}(B_1)$. For any $Z_0=(\rho_0,\theta_0,\xi_0)$ with $\rho_0\in (0,\frac{1}{2})$ and $\abs{\xi_0}<1/2$, $u$ as a function in the lifted holomorphic coordinates are bounded in $C^{2,\alpha}_V$ satisfying
	\begin{equation*}
		\norm{u_{V,i\bar{j}}}_{C^\alpha_V(B_{c_\beta \rho_0}(v_0))}\leq C \norm{u}_{C_\beta^{2,\alpha}(B_1)}
	\end{equation*}
	for some $C>0$ independent of $Z_0$ and $u$.
\end{lem}
\begin{proof}
	First of all, we note that the distance functions involved in the definition of $C_V^\alpha(B_{c_\beta\rho_0}(v_0))$ and $C_\beta^{2,\alpha}(B_1)$ are the same. To see this, we consider two points $(\rho_1,\theta_1,\xi_1)$ and $(\rho_2,\theta_2,\xi_2)$ with
	\begin{equation*}
		\rho_1,\rho_2\in (\rho_0/2,2\rho_0) \quad \mbox{and} \quad \theta_1,\theta_2 \in (\theta_0-0.1, \theta_0+0.1).
	\end{equation*}
	The distance between them with respect to $g_{cone}= d\rho^2+ \rho^2 \beta^2 d\theta^2 + d\xi^2$ is the square root of
	\begin{equation}\label{eqn:distance}
		\abs{\xi_1-\xi_2}^2 + \abs{ (\rho_1 \cos (\beta\theta_1)- \rho_2 \cos (\beta\theta_2))}^2 + \abs{ (\rho_1 \sin (\beta\theta_1)- \rho_2 \sin (\beta\theta_2))}^2.
	\end{equation}
	A nice way to compute this distance is to set $\eta_i=\beta\theta_i$ for $i=1,2$ and notice that $\abs{\eta_1-\eta_2}<0.2$. On the other hand, the $V$-coordinates of $(\rho_1,\theta_1,\xi_1)$ and $(\rho_2,\theta_2,\xi_2)$ are 
	\begin{equation*}
		(\rho_1 e^{\beta \theta_1},\xi_1) \quad \mbox{and} \quad (\rho_2 e^{\beta \theta_2},\xi_2),
	\end{equation*}
	so that the Euclidean distance between them in $B_{c_\beta \rho_0}(v_0)$ is the same as the square root of \eqref{eqn:distance}.

	Next, we study the relation between $u_{V,i\bar{j}}$ with the quantities in (D1-D4). Since
	\begin{equation*}
		u_{V,1\bar{1}}=\abs{z_1}^{2-2\beta} u_{1\bar{1}} \quad \mbox{and} \quad u_{V,k\bar{l}}= u_{k\bar{l}} \quad (k,l\ne 1),
	\end{equation*}
	they appear in (D1-D4) directly and hence are in $C_\beta^\alpha(B_1)$ and (by what we have just proved) in $C_V^\alpha(B_{c_\beta \rho_0}(v_0))$ when restricted to $B_{c_\beta \rho_0}(v_0)$. When $k\ne 1$, we claim that the $C_V^{\alpha}(B_{c_\beta \rho_0}(v_0))$ norm of 
	\begin{equation}\label{eqn:u1k}
		u_{V,1\bar{k}}= \left( \frac{z_1}{\abs{z_1}} \right)^{1-\beta} \abs{z_1}^{1-\beta} u_{1\bar{k}}
	\end{equation}
	is independent of $Z_0$. By (D3) and the first part of the proof, we know the $C_V^\alpha(B_{c_\beta \rho_0}(v_0))$ norm of $\abs{z_1}^{1-\beta} u_{1\bar{k}}$ is independent of $Z_0$. For any two points $V_1=(\rho_1e^{i\beta \theta_1},\xi_1)$ and $V_2=(\rho_2 e^{i\beta \theta_2},\xi_2)$ in $B_{c_\beta \rho_0}(v_0)$, we have
	\begin{equation*}
		\abs{\theta_1-\theta_2}\leq C \abs{V_1-V_2},
	\end{equation*}
	which implies that the $C_V^\alpha(B_{c_\beta \rho_0}(v_0))$ norm of $\left( \frac{z_1}{\abs{z_1}} \right)^{1-\beta}= e^{i(1-\beta)\theta}$ is independent of $Z_0$ so that the claim is proved. 
	
	The proof of the lemma is finished because the same proof also works for $u_{V,k\bar{1}}$.
\end{proof}

\subsection{The complex Monge-Amp\`ere equation}\label{subsec:equation}
The aim of this section is to rewrite the equation
\begin{equation}
	\det (\varphi_{i\bar{j}}) = \frac{e^{\lambda \varphi + h}}{\abs{z_1}^{2-2\beta}}. 
	\label{eqn:MA1}
\end{equation}
into a form that will be useful in the proof of Theorem \ref{thm:main}. More precisely, we want to
\begin{itemize}
	\item put the equation into the polar coordinates;
	\item move anything other than $\tilde{\triangle} \varphi$ (see (P3) above) to the right hand side;
	\item characterize the nonlinear structure of the right hand side in a useful form.
\end{itemize}
It turns out that we need to carry this out not only for $\varphi$ but also for all its tangent derivatives.

Multiplying the first row and the first column of $\varphi_{i\bar{j}}$ by $\abs{z_1}^{1-\beta}$, we obtain from \eqref{eqn:MA1}
\begin{equation}
\det \left(
\begin{array}[]{cc}
	\abs{z_1}^{2(1-\beta)} \varphi_{1\bar{1}} & \abs{z_1}^{1-\beta} \varphi_{1\bar{\xi}} \\
	\abs{z_1}^{1-\beta} \varphi_{\xi \bar{1}} & \varphi_{\xi\bar{\xi}} 
\end{array}
\right)
= e^{\lambda \varphi + h}.
	\label{eqn:MA4}
\end{equation}
Here $\varphi_{1\bar{\xi}}$ is the row vector $(\varphi_{1\bar{2}}, \cdots, \varphi_{1\bar{n}})$, $\varphi_{\xi\bar{1}}$ is the column vector $(\varphi_{2\bar{1}},\cdots,\varphi_{n\bar{1}})$ and $\varphi_{\xi\bar{\xi}}$ is the matrix $\left( \varphi_{i\bar{j}} \right)_{2\leq i,j\leq n}$.

Setting $P= \abs{z_1}^{1-\beta}\partial_{z_1}$ and recalling that we have defined 
\begin{eqnarray*}
	\tilde{\triangle} &=&  \partial_\rho^2 + \frac{1}{\rho}\partial_\rho + \frac{1}{\beta^2 \rho^2} \partial_\theta^2 \\
&=& 4 \abs{z_1}^{2(1-\beta)}\frac{\partial^2}{\partial z_1 \partial \bar{z}_1},
\end{eqnarray*}
we obtain
\begin{equation}
\det\left(
\begin{array}[]{cc}
	1/4 \tilde{\triangle} \varphi & P \varphi_{\bar{\xi}} \\
	\bar{P} \varphi_{\xi} & \varphi_{\xi\bar{\xi}} 
\end{array}
\right)
= \exp(\lambda\varphi+h).
	\label{eqn:MA5}
\end{equation}
Finally, let's expand the left hand side of \eqref{eqn:MA5} by the definition of $\det$ to see
\begin{equation*}\tag{A}
	\det \left(\varphi_{\xi\bar{\xi}} \right) \tilde{\triangle} \varphi = (P\varphi_{\bar{\xi}}\cdot \bar{P} \varphi_\xi)\# F(\varphi_{\xi\bar{\xi}}) + 4 \exp (\lambda \varphi +h)
\end{equation*}
Here $P\varphi_{\bar{\xi}}\cdot \bar{P} \varphi_\xi$ stands for a term $P \varphi_{\bar{j}}\cdot \bar{P} \varphi_i$ for $i,j=2,\cdots,n$, $F(\varphi_{\xi\bar{\xi}})$ is a polynomial (with constant coefficients) of $\varphi_{i\bar{j}} (i,j=2,\cdots,n)$ and $\#$ means a sum of products of $P\varphi_{\bar{\xi}}\cdot \bar{P}\varphi_\xi$ and $F(\varphi_{\xi\bar{\xi}})$. The importance of the fact that the $P$ derivatives come in conjugate pairs can be seen from the following lemma.
\begin{lem}
	\label{lem:pair} Let $f$ and $g$ be any complex valued functions. $Pf \cdot \bar{P} g$ is a quadratic polynomial (with constant coefficients) of $\partial_\rho f, \partial_\rho g, \frac{1}{\rho}\partial_\theta f$ and $\frac{1}{\rho}\partial_\theta g$.
\end{lem}
\begin{proof}
	By the definition of $P$, we have
\begin{equation*}
	P = \frac{1}{2} \left[  (\cos \theta -i \sin \theta) \partial_\rho + \beta^{-1}(-\sin \theta - i \cos \theta) \frac{1}{\rho}\partial_\theta \right].
\end{equation*}
The proof of this lemma is just the following direct computation,
\begin{eqnarray*}
	P f\cdot \bar{P} g &=& \frac{1}{4}\left( (\cos \theta -i \sin \theta)\partial_\rho f +\beta^{-1} (-\sin \theta -i \cos \theta) \frac{1}{\rho}\partial_\theta f \right) \\
	&& \cdot\left( (\cos \theta +i \sin \theta)\partial_\rho g +\beta^{-1} (-\sin \theta +i \cos \theta) \frac{1}{\rho}\partial_\theta g \right) \\
	&=& \frac{1}{4} \left( \partial_\rho f \partial_\rho g +\beta^{-2} (\frac{1}{\rho}\partial_\theta f) (\frac{1}{\rho} \partial_\theta g) +i \beta^{-1}\partial_\rho f (\frac{1}{\rho}\partial_\theta g) - i \beta^{-1}(\frac{1}{\rho}\partial_\theta f) \partial_\rho g\right) .
\end{eqnarray*}
\end{proof}

Next, we take tangential derivatives of the equation. By tangential derivative, we mean $\partial_{z_k}$ and $\partial_{\bar{z}_k}$ for $k=2,\cdots,n$. It turns out that the exact value of $k$ is not important so that we write $\partial_\xi$ and $\partial_{\bar{\xi}}$ for simplicity. For example, when we write
\begin{equation*}
	\varphi_{\xi\xi\bar{\xi}},
\end{equation*}
we mean
\begin{equation*}
	\partial_{z_{k_1}} \partial_{z_{k_2}} \partial_{\bar{z}_{k_3}} \varphi
\end{equation*}
for any $k_1,k_2,k_3=2,\cdots,n$. Moreover, we write $\chi$ for a finite sequence of $\xi$ and/or $\bar{\xi}$ and use a subscript of number to denote the length of the sequence. For example, by $\varphi_{\chi_2}$, we mean one of
\begin{equation*}
	\varphi_{\xi\xi}, \varphi_{\xi\bar{\xi}}, \varphi_{\bar{\xi}\xi}, \varphi_{\bar{\xi}\bar{\xi}}.
\end{equation*}

We shall derive an equation of $\varphi_{\chi}$ similar to (A). The key point is to show that the equation will have a similar right hand side as in (A). We claim that if the length of $\chi$ is $l$, then
\begin{align*}\tag{$A_\chi$}
	\det (\varphi_{\xi\bar{\xi}}) \tilde{\triangle} \varphi_\chi &=  F(\varphi_{\chi_2},\ldots,\varphi_{\chi_{l+2}})\# ( P\varphi_{\chi_a}\cdot \bar{P} \varphi_{\chi_b}) \\
	& +  H(h,\cdots,h_{\chi_l}, \varphi, \cdots, \varphi_{\chi_{l+2}}).
\end{align*}
Here $F$ and $H$ are smooth functions of their arguments and $a,b=1,\cdots,l+1$.
\begin{rem}
	For the rest of the paper, it is enough to know that $F$ and $H$ are smooth functions of $\varphi$ and $h$ and their tangential derivatives.
\end{rem}

The claim follows from direct computation and induction. Note that we may use $F$ and $H$ for different smooth functions in different lines below. Although the exact formula can be obtained, it is irrelevant to us.

Take $\partial_\xi$ of (A) to get
\begin{eqnarray}\label{eqn:A1}
	&& \det (\varphi_{\xi \bar{\xi}}) \tilde{\triangle} \varphi_\xi \\ \nonumber
	&=& - \partial_\xi \det(\varphi_{\xi\bar{\xi}}) \tilde{\triangle} \varphi + \partial_\xi \left( (P \varphi_{\bar{\xi}}\cdot \bar{P}\varphi_\xi) \# F(\varphi_{\xi\bar{\xi}}) + 4 \exp (\lambda \varphi +h)\right).
\end{eqnarray}
Noticing that
\begin{itemize}
	\item $\partial_\xi \det(\varphi_{\xi\bar{\xi}})$ and $\partial_\xi F(\varphi_{\xi\bar{\xi}})$ are smooth functions of $\varphi_{\xi\bar{\xi}}$ and $\varphi_{\xi\xi\bar{\xi}}$;
	\item $\partial_\xi (P\varphi_{\bar{\xi}}\cdot \bar{P}\varphi_\xi)$ is a sum of $P \varphi_{\xi\bar{\xi}}\cdot \bar{P} \varphi_\xi, P\varphi_{\bar{\xi}}\cdot \bar{P}\varphi_{\xi\xi}$;
	\item $\partial_\xi \exp(\lambda\varphi+h)$ is obviously a smooth function of $\varphi$ and $h$ and $\partial_\xi \varphi$, $\partial_{\xi} h$.
	\item as a consequence of (A),
		\begin{equation*}
			\tilde{\triangle} \varphi = \det(\varphi_{\xi\bar{\xi}})^{-1} \left(P\varphi_{\bar{\xi}} \cdot \bar{P} \varphi_\xi \# F + 4 \exp(\lambda\varphi+h)  \right),
		\end{equation*}
\end{itemize}
we conclude that the right hand side of \eqref{eqn:A1} is of the required form in ($A_\chi$) and thus the claim is proved if the length of $\chi$ is one. The general case shall follow by a similar computation which we omit.

By the assumption of the Theorem \ref{thm:main}, $\det (\varphi_{\xi\bar{\xi}})$ is H\"older continuous so that we can rewrite ($A_\chi$) as
\begin{align*}\tag{$B_\chi$}
	\det(\varphi_{\xi\bar{\xi}})(0,\xi) \tilde{\triangle} \varphi_\chi &=  F \# (P\varphi_{\chi_a} \cdot \bar{P} \varphi_{\chi_b})+ H  \\
	 &+ (\det(\varphi_{\xi\bar{\xi}})(0,\xi)- \det(\varphi_{\xi\bar{\xi}})(\rho,\theta,\xi)) \tilde{\triangle} \varphi_\chi. 
\end{align*}
Here $\det(\varphi_{\xi\bar{\xi}})(0,\xi)$ is short for $\det(\varphi_{\xi\bar{\xi}})(0,\theta,\xi)$ since it is independent of $\theta$.

\section{Interior estimates for $\varphi$}\label{sec:interior}
In this section, we prove higher order estimates for the potential function $\varphi$ away from the singular set. It is quite natural that as one gets closer and closer to the singular set, the estimates become worse and worse. Such a phenomenon usually appears as a weighted derivative estimate. It serves as a starting point of the bootstrapping argument in the proof of the main result in this paper.
\begin{thm}\label{thm:interior}
	Let $\varphi$ be a $C^{2,\alpha}_\beta(B_1)$ solution of \eqref{eqn:MA1} as given in Theorem \ref{thm:main}. Then given any $k_1,k_2,k_3 \in \mathbb N\cup \set{0}$, we have
	\begin{equation}\label{eqn:interior}
		\abs{(\rho \partial_\rho)^{k_1} (\partial_\theta)^{k_2} (\nabla_\xi)^{k_3} \varphi}\leq C (k_1,k_2,k_3),\quad \forall \rho\in (0,1/2) \,\text{and}\, \abs{\xi}<1/2, 
	\end{equation}
	for some constant $C(k_1,k_2,k_3)$.
\end{thm}

\begin{rem}
	We remark that the same estimate was proved as the Step 1 in Section 4 of \cite{jeffres2011k}. We still include a complete proof here because (1) we can not follow the scaling argument on page 135 of \cite{jeffres2011k} and (2) since we have assumed Donaldson's estimate, i.e. the H\"older continuity of the complex Hessian of $\varphi$, it looks a bit strange if we use the Evans-Krylov theorem again.
\end{rem}

By Lemma \ref{lem:tildeW}, it suffices to show that for any $Z_0=(\rho_0,\theta_0,\xi_0)$ with $\rho_0\ne 0$, any derivatives of $(\nabla_\xi)^{k_3}\varphi$ in scaled lifted holomorphic coordinates around $Z_0$ are uniformly bounded. Hence in what follows, we have $Z_0$ fixed and work in the $\tilde{V}$-coordinates around it.

The proof is still a scaling argument. However, we find it necessary to study first the scaling of the complex Hessian of $\varphi$, which may be and should be regarded as a K\"ahler metric and which satisfies an elliptic system. By using some well known estimates for this elliptic system, we show that the complex Hessian has the desired weighted estimates, or equivalently, it is reasonably bounded in the scaled lifted holomorphic coordinates (see Lemma \ref{lem:tildeW}). This is the purpose of Lemma \ref{lem:goodmetric} and Lemma \ref{lem:goodtangent}, which we prove in the next subsection.

\subsection{Metric (complex Hessian) in (scaled) lifted holomorphic coordinates}
First, we put \eqref{eqn:MA1} in the lifted holomorphic coordinates around $Z_0=(\rho_0,\theta_0,\xi_0)$. For that purpose, we multiply the first column of $\varphi_{i\bar{j}}$ by $\bar{z}_1^{1-\beta}$ and the first row by $z_1^{1-\beta}$ and notice that $\partial_{v_1}=z_1^{1-\beta} \partial_{z_1}$ and  $\partial_{\bar{v}_1}=\bar{z}_1^{1-\beta} \partial_{\bar{z}_1}$ to get
\begin{equation}\label{eqn:MAinV}
	\log \det (\varphi_{V,i\bar{j}}) = \lambda \varphi +h.	
\end{equation}
The assumptions in Theorem \ref{thm:main} implies
\begin{equation}\label{eqn:varphic0}
	\frac{1}{c}\delta_{ij} \leq \varphi_{V,i\bar{j}}\leq c \delta_{ij},  \quad \mbox{by (S3)}
\end{equation}
and
\begin{equation}\label{eqn:Gcalpha}
	\norm{\varphi_{V,i\bar{j}}}_{C_V^\alpha(B_{c_\beta \rho_0}(v_0))}\leq C, \quad \mbox{by Lemma \ref{lem:metricinV}}.
\end{equation}
Donaldson's Schauder estimate and \eqref{eqn:MA1} show that all tangent derivatives of $\varphi$ are bounded in $C_\beta^{2,\alpha}$ space so that Lemma \ref{lem:metricinV} again implies that
\begin{equation}\label{eqn:Xigc0}
	\norm{ (\nabla_{V,\xi}^l \varphi)_{V,i\bar{j}}}_{C_V^\alpha(B_{c_\beta \rho_0}(v_0))}\leq C(l).
\end{equation}
Here $l$ is any natural number and $C(l)$ depends on $l$ but not on $Z_0$.

Next, we move to work in scaled lifted holomorphic coordinates by setting
\begin{equation}\label{eqn:setting}
	\tilde{\varphi}(\tilde{v})=\rho_0^{-2} \varphi(v_0+\tilde{v}\rho_0). 
\end{equation}
\eqref{eqn:varphic0}, \eqref{eqn:Gcalpha} and \eqref{eqn:Xigc0} imply respectively 
\begin{equation}\label{eqn:Gc0}
	\frac{1}{c}\delta_{ij} \leq \tilde{\varphi}_{\tilde{V},i\bar{j}}\leq c \delta_{ij},
\end{equation}
\begin{equation}\label{eqn:gcalpha}
	\norm{\tilde{\varphi}_{\tilde{V},i\bar{j}}}_{C_{\tilde{V}}^\alpha(B_{c_\beta})}\leq C
\end{equation}
and for $l=1,2,\cdots$,
\begin{equation}\label{eqn:xigc0}
	\norm{ (\nabla_{\tilde{V},\xi}^l \tilde{\varphi})_{\tilde{V},i\bar{j}}}_{C^0(B_{c_\beta})}\leq C \rho_0^{l}.
\end{equation}
If we denote 
$$g_{i\bar{j}}:=\tilde{\varphi}_{\tilde{V},i\bar{j}}, \quad \text{for} \, i,j=1,\cdots,n$$
then $g_{i\bar{j}}$ is a good metric with $C^\alpha$ control in $B_{c_\beta}$ by \eqref{eqn:Gc0} and \eqref{eqn:gcalpha}. 

In order to obtain higher order estimates, we derive an elliptic system satisfied by $g_{i\bar{j}}$. By \eqref{eqn:setting}, taking $\partial_{\tilde{v}_i}\partial_{\bar{\tilde{v}}_j}$ of \eqref{eqn:MAinV} and using $\partial_k g_{i\bar{j}}=\partial_i g_{k\bar{j}}$ (which is the K\"ahler condition if we regard $g_{i\bar{j}}$ as a K\"ahler metric, or the switching order of derivatives if we regard $g_{i\bar{j}}$ as the complex Hessian), we get
\begin{equation}
	\triangle_g g_{i\bar{j}} - g^{k\bar{m}} g^{n\bar{l}} \pfrac{g_{i\bar{m}}}{\tilde{v}_n} \pfrac{g_{k\bar{j}}}{\bar{\tilde{v}}_{{l}}} = \lambda\rho_0^2 g_{i\bar{j}} + h_{\tilde{V},i\bar{j}},
	\label{eqn:gij}
\end{equation}
where $\triangle_g = g^{i\bar{j}}\frac{\partial^2}{\partial \tilde{v}_i \partial \bar{\tilde{v}}_j}$.
If we neglect the dependence of $\triangle_g$ on $g_{i\bar{j}}$ itself, this is a quasilinear elliptic system whose nonlinear term is quadratic in the gradient of the unknown. By Corollary \ref{cor:h}, the only non-homogeneous term in \eqref{eqn:gij} satisfies
\begin{equation}\label{eqn:goodh}
	\norm{h_{\tilde{V},i\bar{j}}}_{C_{\tilde{V}}^k(B_{c_\beta})}\leq C(k) \qquad \text{independent of} \quad Z_0.
\end{equation}

There is a rich theory about elliptic systems of this kind. In particular, the theorems in Chapter VI of \cite{giaquinta1983multiple} imply that there exists some $\alpha'>0$ such that for any $\eta>0$ 
\begin{equation}\label{c1alpha}
	\norm{g_{i\bar{j}}}_{C_{\tilde{V}}^{1,\alpha'}(B_{c_\beta-\eta})}\leq C
\end{equation}
for some constant $C$ depending on $\eta$, $\alpha'$, $c$ in \eqref{eqn:Gc0} and the constant in \eqref{eqn:gcalpha}.
For the reader's convenience, we include a complete proof of this in the Appendix and what we have just used is Lemma \ref{lem:gia} there.

With \eqref{c1alpha} and \eqref{eqn:goodh}, the usual Schauder estimates applied to \eqref{eqn:gij} give higher order estimates of $g_{i\bar{j}}$. That is, 
\begin{lem}
	There are constants $C(k)$ independent of $Z_0$ such that
	\begin{equation*}
		\norm{g_{i\bar{j}}}_{C^k_{\tilde{V}}(B_{3c_\beta/4})} \leq C(k).
	\end{equation*}
	\label{lem:goodmetric} 
\end{lem}

For the proof of Theorem \ref{thm:interior}, we also need estimates for $\nabla^l_{\tilde{V},\xi} g_{i\bar{j}}$. 
\begin{lem}For any $l\in \mathbb N$, there are constants $C(k,l)$ independent of $Z_0$ such that
	\begin{equation}
		\norm{ (\nabla_{\tilde{V},\xi})^l g_{i\bar{j}}}_{C^k_{\tilde{V}}(B_{c_\beta/2})} \leq C(k,l)\rho^l.
		\label{eqn:goodtangent}
	\end{equation}
	\label{lem:goodtangent}
\end{lem}

\begin{proof} Pick any sequence of $\eta_l$ satisfying
	\begin{equation*}
		0<\eta_1<\eta_2<\cdots < \frac{1}{4}c_\beta.
	\end{equation*}
	We will prove
	\begin{equation}
		\norm{ (\nabla_{\tilde{V},\xi})^l g_{i\bar{j}}}_{C^k_{\tilde{V}}(B_{3c_\beta/4-\eta_l})} \leq C(k,l)\rho^l,
		\label{eqn:goodeta}
	\end{equation}
	which is obviously stronger than \eqref{eqn:goodtangent}.
	For the proof below, we allow the constants to depend on this particular choice of $\eta_l$.

	Taking $\nabla_{\tilde{V},\xi}$ of \eqref{eqn:gij} yields
	\begin{equation}\label{eqn:extra}
		\begin{split}
		& \triangle_g (\nabla_{\tilde{V},\xi} g_{i\bar{j}}) + G_1(g,\nabla_{\tilde{V}} g, \nabla^2_{\tilde{V}} g) \# \nabla_{\tilde{V},\xi} g + G_2(g,\nabla_{\tilde{V}} g) \# \nabla_{\tilde{V}} (\nabla_{\tilde{V},\xi} g)\\
		=& (\nabla_{\tilde{V},\xi} h)_{\tilde{V},i\bar{j}},
		\end{split}
	\end{equation}
	which is a linear elliptic system in which we take $\nabla_{\tilde{V},\xi}g_{i\bar{j}}$ as the unknowns. By Lemma \ref{lem:goodmetric}, $g_{i\bar{j}}$, $G_1$ and $G_2$ are bounded in any $C^k$ norm on $B_{3c_\beta/4}$ so that the Schauder estimate gives 
	\begin{equation}\label{eqn:schauder}
		\norm{\nabla_{\tilde{V},\xi} g}_{C^{k+2,\alpha}_{\tilde{V}}(B_{(3c_\beta/4-\eta_1)})}\leq C \left(  \norm{\nabla_{\tilde{V},\xi} g}_{C^{0}(B_{c_\beta})} + \norm{\nabla_{\tilde{V},\xi} h}_{C^{k+2,\alpha}_{\tilde{V}}(B_{c_\beta})}\right).
	\end{equation}
	Since $h$ is a smooth function, we may apply Corollary \ref{cor:h} to $\nabla_{V,\xi} h$ to get 
	\begin{equation}\label{eqn:hh}
		\norm{\nabla_{V,\xi} h}_{C_{\tilde{V}}^{k+2,\alpha}(B_{c_\beta})}\leq C
	\end{equation}
	which is equivalent to
	\begin{equation}\label{eqn:h1}
		\norm{\nabla_{\tilde{V},\xi} h}_{C^{k+2,\alpha}_{\tilde{V}}(B_{c_\beta})}\leq C\rho_0.
	\end{equation}
	Combining \eqref{eqn:xigc0}, \eqref{eqn:h1} and \eqref{eqn:schauder} proves \eqref{eqn:goodeta} in case $l=1$.
\vskip 2mm

We now prove by induction that for $l\geq 2$, 
 $(\nabla_{\tilde{V},\xi})^l g_{i\bar{j}}$ satisfies 
\begin{equation}
		\begin{split}
			& \triangle_g (\nabla_{\tilde{V},\xi})^{l} g_{i\bar{j}} + G_1(g,\nabla_{\tilde{V}} g, \nabla^2_{\tilde{V}} g) \# (\nabla_{\tilde{V},\xi})^{l} g + G_2(g,\nabla_{\tilde{V}} g) \# \nabla_{\tilde{V}} ( (\nabla_{\tilde{V},\xi})^{l} g)\\
			=&  \left((\nabla_{\tilde{V},\xi})^l h  \right)_{\tilde{V},i\bar{j}} +  \mathcal H_{l},
		\end{split}
	\label{eqn:extral}
\end{equation}
where $\mathcal H_l$ is the sum of terms of the following form 
\begin{equation*}
	G(g, \nabla_{\tilde{V}} g, \nabla^2_{\tilde{V}} g) (\nabla_{\tilde{V}})^{a_1} (\nabla_{\tilde{V},\xi})^{b_1} g \cdot \cdots \cdot (\nabla_{\tilde{V}})^{a_s} (\nabla_{\tilde{V},\xi})^{b_s} g.
\end{equation*}
Here $2\leq s\leq l$ is an integer, for each $i=1,\cdots,s$, $a_i$ takes only three possible values, namely, $0,1$ or $2$ and $b_i$ is a positive integer strictly smaller than $l$, satisfying
\begin{equation*}
	b_1+\cdots+b_s=l.
\end{equation*}
Moreover, $G$, depending on $a_i$ and $b_i$, is a smooth function of its arguments. 

Obviously, when $l=1$, $\mathcal H_l$ must be zero and \eqref{eqn:extral} is exactly \eqref{eqn:extra}. Assume that \eqref{eqn:extral} holds for $l$. Direct computation gives
\begin{equation}
		\begin{split}
			& \triangle_g (\nabla_{\tilde{V},\xi})^{l+1} g_{i\bar{j}} + G_1(g,\nabla_{\tilde{V}} g, \nabla^2_{\tilde{V}} g) \# (\nabla_{\tilde{V},\xi})^{l+1} g + G_2(g,\nabla_{\tilde{V}} g) \# \nabla_{\tilde{V}} ( (\nabla_{\tilde{V},\xi})^{l+1} g)\\
			=&  \left((\nabla_{\tilde{V},\xi})^{l+1} h  \right)_{\tilde{V},i\bar{j}} +  \nabla_{\tilde{V},\xi} \mathcal H_{l}  + \nabla_{\tilde{V},\xi} g^{\cdot\cdot} \# (\nabla_{\tilde{V}})^2 (\nabla_{\tilde{V},\xi})^l g  \\
			& + \nabla_{\tilde{V},\xi} \left( G_1(g,\nabla_{\tilde{V}} g, \nabla^2_{\tilde{V}} g)\right) \# (\nabla_{\tilde{V},\xi})^{l} g + \nabla_{\tilde{V},\xi} \left( G_2(g,\nabla_{\tilde{V}} g)\right) \# \nabla_{\tilde{V}} ( (\nabla_{\tilde{V},\xi})^{l} g).
		\end{split}
	\label{eqn:extral1}
\end{equation}
Here $g^{\cdot\cdot}$ stands for one of $g^{i\bar{j}}$. It follows from the chain rule and the definition of $\mathcal H_l$ that \eqref{eqn:extral} holds for $l+1$ as well.

With \eqref{eqn:extral}, we can finish the proof of this lemma. Similar to \eqref{eqn:h1}, we have
	\begin{equation}\label{eqn:h2}
		\norm{(\nabla_{\tilde{V},\xi})^l h}_{C^{k+2,\alpha}_{\tilde{V}}(B_{c_\beta})}\leq C\rho_0^l.
	\end{equation}
	Moreover, by induction, if \eqref{eqn:goodeta} is proved for $l'<l$,
	\begin{equation}\label{eqn:h3}
		\norm{\mathcal H_l}_{C^{k,\alpha}_{\tilde{V}}(B_{c_\beta-\eta_{l-1}})}\leq C\rho_0^l.
	\end{equation}
	\eqref{eqn:goodeta} for $l$ follows from the Schauder estimate of \eqref{eqn:extral} by \eqref{eqn:h2}, \eqref{eqn:h3} and \eqref{eqn:xigc0}.
\end{proof}
\vskip 1cm 
With these preparations, we move on to the proof of Theorem \ref{thm:interior}.
\subsection{The case $k_3=0,1$}

While we get good control for the complex Hessian of $\tilde{\varphi}(\tilde{v})$ (see \eqref{eqn:gcalpha}), the $C^0$ norm of $\tilde{\varphi}$ is made worse by a multiple of $\rho_0^{-2}$, i.e.,
\begin{equation}
	\norm{\tilde{\varphi}(\tilde{v})}_{C^0(B_{c_\beta})}\leq C \rho_0^{-2}
	\label{eqn:tildeWC0}
\end{equation}
and the $C^0$ norm of the first order derivative of $\tilde{\varphi}$ becomes
\begin{equation}
	\norm{\nabla_{\tilde{V}} \tilde{\varphi}(\tilde{v})}_{C^0(B_{c_\beta})}\leq C \rho_0^{-1}.
	\label{eqn:tildeWC1}
\end{equation}

To prove Theorem \ref{thm:interior} for the case $k_3=0,1$, we consider the equation satisfied by $\nabla_{\tilde{V}} \tilde{\varphi}$,
\begin{equation}\label{eqn:dwpsi}
	\triangle_g (\nabla_{\tilde{V}} \tilde{\varphi}) =\lambda \rho_0^2 \nabla_{\tilde{V}} \tilde{\varphi} + \nabla_{\tilde{V}} h,
\end{equation}
which is obtained by taking $\tilde{V}$-derivative of the Monge-Amp\`ere equation satisfied by $\tilde{\varphi}$ (see \eqref{eqn:MAinV} and \eqref{eqn:setting}) 
\begin{equation*}
	\log \det (\tilde{\varphi}_{\tilde{V},i\bar{j}}) =\lambda \rho_0^2 \tilde{\varphi} + h.
\end{equation*}
By Lemma \ref{lem:goodmetric} and Corollary \ref{cor:h}, any $C^{k,\alpha}_{\tilde{V}}$ norm of the metric $g$ is bounded and 
\begin{equation}\label{temph}
	\norm{\nabla_{\tilde{V}} h}_{C^{k,\alpha}_{\tilde{V}}(B_{c_\beta/2})}\leq C
\end{equation}
so that the Schauder estimate applied to \eqref{eqn:dwpsi} gives
\begin{equation}\label{eqn:smallk3}
	\norm{\nabla_{\tilde{V}}\tilde{\varphi}}_{C^{k+2,\alpha}_{\tilde{V}}(B_{c_\beta/4})}\leq C \left( \norm{\nabla_{\tilde{V}}\tilde{\varphi}}_{C^0(B_{c_\beta/2})} + \norm{\nabla_{\tilde{V}} h}_{C^{k,\alpha}_{\tilde{V}}(B_{c_\beta/2})} \right) \leq C \rho_0^{-1},
\end{equation}
where we have used \eqref{eqn:tildeWC1} and \eqref{temph} in the last inequality above.
By \eqref{eqn:setting} and \eqref{eqn:smallk3},
\begin{equation}\label{eqn:k30}
	\norm{\varphi}_{C^{k,\alpha}_{\tilde{V}}(B_{c_\beta/4})}\leq C \quad \mbox{for} \quad k\geq 1,
\end{equation}
which is exactly what we need for the case $k_3=0$ after a further translation to the polar coordinates (see Lemma \ref{lem:tildeW}). 

For $k_3=1$, it suffices to notice that \eqref{eqn:smallk3} implies
\begin{equation*}
	\norm{\nabla_{V,\xi} \varphi}_{C^{k,\alpha}_{\tilde{V}}(B_{c_\beta/4})}\leq C \quad \mbox{for any} \quad k \geq 0. 
\end{equation*}

\subsection{The case $k_3>1$}
While the basic strategy of the proof in this case is similar, we need to use some extra information provided by the tangential regularity of $\varphi$:
\begin{equation}\label{eqn:k32}
	\norm{(\nabla_{\tilde{V},\xi})^{k_3} \tilde{\varphi}}_{C^0(B_{c_\beta})}\leq C \rho_0^{k_3-2},
\end{equation}
which follows from the boundedness of $\nabla^{k_3}_{V,\xi} \varphi$ and \eqref{eqn:setting}.
By \eqref{eqn:setting} again and Lemma \ref{lem:tildeW}, the proof of Theorem \ref{thm:interior} for the case $k_3>1$ is reduced to the claim that
\begin{equation}\label{eqn:weclaim}
	\norm{(\nabla_{\tilde{V},\xi})^{k_3} \tilde{\varphi}}_{C^{l,\alpha}_{\tilde{V}}(B_{c_\beta/4})}\leq C(l) \rho_0^{k_3-2} \qquad \text{for} \quad l\in \mathbb N.
\end{equation}


The rest of this section is devoted to the proof of \eqref{eqn:weclaim}, which is an induction on $k_3$.
We first notice that \eqref{eqn:weclaim} for the case $k_3=1$ is a special case of \eqref{eqn:smallk3}. For $k_3=2$, we take one more $\nabla_{\tilde{V},\xi}$ of the following equation (see \eqref{eqn:dwpsi})
\begin{equation*}
	\triangle_g (\nabla_{\tilde{V},\xi} \tilde{\varphi}) =\lambda \rho_0^2 \nabla_{\tilde{V},\xi} \tilde{\varphi}+ \nabla_{\tilde{V},\xi} h
\end{equation*}
to get
\begin{equation}\label{eqn:2k3}
	\triangle_{g} (\nabla_{\tilde{V},\xi}^2 \tilde{\varphi}) =\lambda \rho_0^2 \nabla_{\tilde{V},\xi}^2 \tilde{\varphi} + (\nabla_{\tilde{V},\xi} \tilde{\varphi})_{\tilde{V},i\bar{j}} \# \nabla_{\tilde{V},\xi} g^{i\bar{j}} + \nabla^2_{\tilde{V},\xi} h.
\end{equation}
The idea is to take this as a linear equation of $\nabla^2_{\tilde{V},\xi} \tilde{\varphi}$ defined on $B_{c_\beta}$ and to apply the Schauder estimate. For this purpose, we check that: (a) the coefficients of $\triangle_g$ are known to be good by Lemma \ref{lem:goodmetric}; (b) by \eqref{eqn:k32}, the $C^0(B_{c_\beta})$ norm of $\nabla^2_{\tilde{V},\xi} \tilde{\varphi}$ is bounded by a constant independent of $\rho_0$, or a constant multiple of $\rho_0^{k_3-2}$ since $k_3=2$; (c) the non-homogeneous term is (by switching the order of derivatives) 
\begin{equation*}
	(\nabla_{\tilde{V},\xi} \tilde{\varphi})_{\tilde{V},i\bar{j}} \# \nabla_{\tilde{V},\xi} g^{i\bar{j}} + \nabla^2_{\tilde{V},\xi} h= \nabla_{\tilde{V},\xi} g_{i\bar{j}} \# \nabla_{\tilde{V},\xi} g^{i\bar{j}}+ \nabla^2_{\tilde{V},\xi} h,
\end{equation*}
where the first term is estimated by Lemma \ref{lem:goodtangent}
\begin{equation}\label{eqn:non1}
	\norm{\nabla_{\tilde{V},\xi} g_{i\bar{j}} \# \nabla_{\tilde{V},\xi} g^{i\bar{j}}}_{C^{k,\alpha}_{\tilde{V}}(B_{c_\beta/2})} \leq C \rho_0^2 \leq C
\end{equation}
and the second term by Corollary \ref{cor:h} (applying to $\nabla^2_\xi h$)
\begin{equation}\label{eqn:non2}
	\norm{\nabla^2_{\tilde{V},\xi} h}_{C^{k,\alpha}_{\tilde{V}}(B_{c_\beta/2})} \leq C \rho_0^2 \leq C.
\end{equation}
By (a)-(c) above, applying the Schauder estimate to \eqref{eqn:2k3} concludes the proof of \eqref{eqn:weclaim} for $k_3=2$. 

For $k_3=3$, we take one more $\nabla_{\tilde{V},\xi}$ to \eqref{eqn:2k3} and find a similar equation with a more complicated non-homogeneous term. The key point is again that the $C^{k,\alpha}_{\tilde{V}}$ norm of this non-homogeneous term is now bounded by $ C\rho_0 = C \rho_0^{k_3-2}$ by the same reason as above so that we can obtain \eqref{eqn:weclaim} (for $k_3=3$) from \eqref{eqn:k32} by using the Schauder estimate. 

We repeat this argument to see that \eqref{eqn:weclaim} for any $k_3$ (hence Theorem \ref{thm:interior}) holds.

\section{Expansion}
In this section, we prove Theorem \ref{thm:main} and Corollary \ref{cor:main}. 

\subsection{Formal consideration}
For any $q\in \Real$, define
\begin{equation*}
	\mathcal T=\set{\rho^{2j+\frac{k}{\beta}}\cos l\theta,\rho^{2j+\frac{k}{\beta}}\sin l\theta|\, j,k,l\in \mathbb N\cup\set{0},\frac{k-l}{2}\in \mathbb N\cup \set{0} }.
\end{equation*}
If we write $Span(\mathcal T)$ for the set of finite linear combinations of functions in $\mathcal T$, it is easy to check that $Span(\mathcal T)$ is multiplicatively closed. In \cite{yin2016}, the author proved that the Ricci flow solution (on surface) has an expansion using terms in $\mathcal T$. Please note that $\beta$ in this paper is $\beta+1$ in that paper. The key observation there is that if the unknown functions are assumed to have an expansion using functions in $\mathcal T$, then by the Ricci flow equation, we know the $\tilde{\triangle}$ of the unknown function has the same expansion, which is a consequence of the nonlinear structure of the Ricci flow equation and the fact that $Span(\mathcal T)$ is multiplicatively closed. Moreover, the fact that $Span(\mathcal T)$ is also closed under the application of $\tilde{\triangle}^{-1}$ allows us to prove that the unknown function has the desired expansion using terms in $\mathcal T$ by an induction argument. 

Before we proceed, let's show how the complex Monge-Amp\`ere equation is more complicated, which forces us to consider terms involving $\log \rho$. Recall the equation of $\varphi$ and $\varphi_\chi$
\begin{align*}\tag{$B_\chi$}
	\det(\varphi_{\xi\bar{\xi}})(0,\xi) \tilde{\triangle} \varphi_\chi &=  F \# (P\varphi_{\chi_a} \cdot \bar{P} \varphi_{\chi_b})+ H  \\
	 &+ (\det(\varphi_{\xi\bar{\xi}})(0,\xi)- \det(\varphi_{\xi\bar{\xi}})(\rho,\theta,\xi)) \tilde{\triangle} \varphi_\chi,
\end{align*}
where $F$ and $H$ are smooth functions of $\varphi$ and $h$ and their tangent derivatives.

Now let's look at the right hand side of ($B_{\chi}$). Since $h$ and its tangent derivatives are smooth functions of $z_1$ and 
\begin{equation*}
	z_1= \beta^{1/\beta} \left( \rho^{\frac{1}{\beta}}\cos \theta + \sqrt{-1} \rho^{\frac{1}{\beta}}\sin \theta\right),
\end{equation*}
they should have an expansion in $\mathcal T$, which is later proved in Lemma \ref{lem:smooth}. If all of $\varphi_\chi$'s have expansion involving terms in $\mathcal T$, then formally, we expect that $F$ and $H$ also have expansion involving terms in $\mathcal T$. The problem is the product $P \varphi_{\chi_a}\cdot \bar{P} \varphi_{\chi_b}$.

\begin{rem}
	In the rest of this paper, the functions in $\mathcal T$ shall frequently appear in the argument. For simplicity, we discuss the $\cos$ term only and understand that a minor modification of the proof works for the $\sin$ term.
\end{rem}

By the definition of $\mathcal T$ and Lemma \ref{lem:pair}, $P\varphi_{\chi_a} \cdot \bar{P} \varphi_{\chi_b}$ is, up to some error, a sum of
\begin{equation}\label{eqn:badterm}
	\rho^{2j+\frac{k}{\beta}-2}\cos l\theta \quad \mbox{with} \quad j+k\geq 2, \frac{k-l}{2}\in \mathbb N\cup\set{0}.
\end{equation}
Some terms of \eqref{eqn:badterm} which may appear in the expansion of the right hand side of ($B_\chi$) cause trouble when we apply $\tilde{\triangle}^{-1}$ to the right hand side. To see this, we compute
\begin{eqnarray}\label{eqn:laplace}
	\tilde{\triangle} \rho^{\sigma+2} \cos l\theta &=&  (\partial_\rho^2 + \frac{1}{\rho}\partial_\rho - \frac{l^2}{\rho^2(\beta)^2}) \rho^{\sigma+2} \cos l\theta \\
	&=& ( (\sigma+2)^2 - \frac{l^2}{(\beta)^2}) \rho^\sigma \cos l\theta,
\end{eqnarray}
which implies that:
if $\rho^{\frac{k}{\beta}-2}\cos k\theta$ (a special case in \eqref{eqn:badterm}) appears in the expansion of the right hand side of ($B_\chi$), then we can not find a match in the expansion of the left hand side of ($B_\chi$).

This forces us to consider more terms than $\mathcal T$. The simple computation
\begin{equation*}
	\tilde{\triangle} (\rho^{\frac{k}{\beta}}(\log \rho) \cos k\theta) = \frac{2k}{\beta} \rho^{\frac{k}{\beta}-2}\cos k\theta
\end{equation*}
motivates us to include (in the expansion of $\varphi$ and $\varphi_\chi$)
\begin{equation*}
	\rho^{\frac{k}{\beta}}\log \rho \cos k\theta \quad \mbox{for} \quad k\geq 2.
\end{equation*}
Here $k\geq 2$ is a consequence of the $j+k\geq 2$ in \eqref{eqn:badterm}.

\vskip 2mm
Concluding the formal discussion above, we define 
\begin{defn}\label{defn:Tlog}
$\mathcal T_{\log}$ is defined to be the set of
\begin{equation*}
	\rho^{2j+\frac{k}{\beta}}(\log \rho)^m \cos l\theta,\quad
	\rho^{2j+\frac{k}{\beta}}(\log \rho)^m \sin l\theta
\end{equation*}
satisfying
\begin{enumerate}[(1)]
	\item $k,j,l,m=0,1,2,\cdots$;
	\item $\frac{k-l}{2}\in \mathbb N\cup\set{0}$;
	\item $m\leq \max \set{0,k-1}$. 
\end{enumerate}
\end{defn}
It is trivial to check that $Span(\mathcal T_{\log})$ is multiplicatively closed. We shall see in later proofs that (3) plays a subtle role in balancing the solvability of $\tilde{\triangle}$ and the nonlinear structure of the right hand side of ($B_\chi$).

\begin{lem}\label{lem:solveLaplace}
	Except the case $k=l=m+1$ and $j=0$ and the case $k=l=j=0$, for any $v=\rho^{2j+\frac{k}{\beta}}(\log \rho)^m \cos l\theta$ in $\mathcal T_{\log}$, we have $u\in Span(\mathcal T_{\log})$ such that
	\begin{equation*}
		\tilde{\triangle} u = \rho^{-2} v.
	\end{equation*}
\end{lem}

\begin{proof}
	The proof relies on the following computation
	\begin{eqnarray}\label{eqn:log}
		\tilde{\triangle} \rho^{\sigma}(\log \rho)^m \cos l \theta &=&  \left( \sigma^2-\frac{l^2}{\beta^2} \right) \rho^{\sigma-2} (\log \rho)^m \cos l\theta \\ \nonumber
		&& + 2\sigma m \rho^{\sigma-2}(\log \rho)^{m-1}\cos l\theta \\ \nonumber
		&& + m(m-1) \rho^{\sigma-2} (\log \rho)^{m-2} \cos l\theta.
	\end{eqnarray}

	(1) If either $j\ne 0$ or $k\ne l$, we prove the lemma by induction on $m$ as follows. Notice that in this case, we always have $2j+\frac{k}{\beta}> \frac{l}{\beta}$ because $k\geq l$ as required in the definition of $\mathcal T_{\log}$.

	When $m=0$, $v=\rho^\sigma\cos l\theta$ with $\sigma=2j+\frac{k}{\beta}\ne \frac{l}{\beta}$. \eqref{eqn:log} implies that $\tilde{\triangle} v$ is a multiple of $\rho^{-2} v$, which proves the lemma in this case. 

	Assume that the case $m\leq m_0$ for some $m_0< \max\set{0,k-1}$ is proved and that $m_0+1\leq \max\set{0, k-1}$ such that $v=\rho^{2j+\frac{k}{\beta}}(\log \rho)^{m_0+1}\cos l\theta$ is in $\mathcal T_{\log}$. By taking $u_1= v$, \eqref{eqn:log} again implies that $\tilde{\triangle} u_1$ is a multiple of $\rho^{-2} v$ up to a linear combination of
	\begin{equation*}
		\rho^{2j+\frac{k}{\beta}-2}(\log \rho)^{m_0-1}\cos l\theta\quad (\text{if } m_0\geq 1) \quad \mbox{and} \quad \rho^{2j+\frac{k}{\beta}-2} (\log \rho)^{m_0} \cos l\theta.
	\end{equation*}
	By the induction hypothesis, there exists $u_2,u_3\in Span(\mathcal T_{\log})$ such that the above two terms are $\tilde{\triangle} u_2$ and $\tilde{\triangle} u_3$ respectively. The desired $u$ is then the linear combination of $u_1,u_2$ and $u_3$.

	(2) For the rest of the proof, we assume $j=0$ and $k=l\ne 0$. In this case, $\sigma=\frac{l}{\beta}$ and the first term in the right hand side of \eqref{eqn:log} vanishes. Notice that there is nothing to prove for $k=l=1$, so we assume that $k\geq 2$ and prove by induction that the lemma holds for $m=0,1,\cdots,k-2$.
	
	When $m=0$, i.e. $v=\rho^{\frac{k}{\beta}}\cos k\theta$, we can take $u=\frac{\beta}{2k}\rho^{\frac{k}{\beta}}(\log \rho)\cos k\theta$ in $Span(\mathcal T_{\log})$, because as a special case of \eqref{eqn:log},
	\begin{equation*}
		\tilde{\triangle} \rho^{\frac{k}{\beta}}(\log \rho) \cos k\theta = \frac{2k}{\beta} \rho^{\frac{k}{\beta}-2} \cos k\theta.
	\end{equation*}

	Assume that lemma is proved for $m\leq m_0$ and that $m_0+1\leq k-2$. Let $v=\rho^{\frac{k}{\beta}}(\log \rho)^{m_0+1} \cos k\theta$ and take $u_1= \rho^{\frac{k}{\beta}}(\log \rho)^{m_0+2} \cos k\theta$ in $\mathcal T_{\log}$. By \eqref{eqn:log}, $\tilde{\triangle} u_1$ is a multiple of $\rho^{-2} v$ up to a multiple of
	\begin{equation*}
		\rho^{\frac{k}{\beta}-2} (\log \rho)^{m_0} \cos k\theta,
	\end{equation*}
	which is $\tilde{\triangle} u_2$ for some $u_2\in Span(\mathcal T_{\log})$ by the induction hypothesis. Again, the desired $u$ is a linear combination of $u_1$ and $u_2$.
\end{proof}

Motivated by this lemma, we define
\begin{defn}
	\label{defn:Trhs}
\begin{equation}\label{eqn:Trhs}
	\mathcal T_{rhs}=\left( \rho^{-2}\mathcal T_{\log}  \right) \setminus \set{\rho^{-2}, \rho^{\frac{k}{\beta}-2}(\log \rho)^{k-1} \cos k\theta |\,k\in \mathbb N},
\end{equation}
where $\rho^{-2} \mathcal T_{\log}$ is the set of $\rho^{-2} v$ for each $v\in \mathcal T_{\log}$. 
\end{defn}
A simple observation is that $\mathcal T_{\log}\subset \mathcal T_{rhs}$. With \eqref{eqn:Trhs}, Lemma \ref{lem:solveLaplace} can be formulated as
\begin{lem}
	\label{lem:rhs}
	For each $v\in Span(\mathcal T_{rhs})$, there is $u$ in $Span(\mathcal T_{\log})$ such that
	\begin{equation*}
		\tilde{\triangle} u =v.
	\end{equation*}
\end{lem}

Besides its connection with Lemma \ref{lem:solveLaplace}, the definition of $\mathcal T_{rhs}$ in \eqref{eqn:Trhs} is important also in the analysis of the structure of the right hand side of ($B_\chi$). Briefly speaking, we will show in Section \ref{subsec:proof} that if $\varphi_\chi$ has expansion using functions in $\mathcal T_{\log}$, then the right hand side of ($B_{\chi}$) has an expansion in $\mathcal T_{rhs}$. This explains the subscript in the notation. For that purpose, we shall need the following lemma. 
\begin{lem}
	\label{lem:logrhs} (1) If $\eta_1\in Span(\mathcal T_{\log})$ and $\eta_2\in Span(\mathcal T_{rhs})$, then $\eta_1\cdot \eta_2 \in Span(\mathcal T_{rhs})$.

	(2) If $\eta_1$ and $\eta_2$ are in $Span(\mathcal T_{\log})$, but neither of them is constant function, then $\rho^{-2} \eta_1 \cdot \eta_2 \in Span(\mathcal T_{rhs})$. 
\end{lem}
\begin{proof}
	(1)

Assume without loss of generality that
\begin{equation*}
	\eta_1=\rho^{2j_1+\frac{k_1}{\beta}}(\log \rho)^{m_1} \cos l_1 \theta \quad \text{and} \quad \eta_2=\rho^{-2+ 2j_2+\frac{k_2}{\beta}}(\log \rho)^{m_2} \cos l_2 \theta,
\end{equation*}
where $k_i,j_i,m_i,l_i$ are subject to the restrictions of Definition \ref{defn:Tlog} and Definition \ref{defn:Trhs} respectively. Direct computation gives 
\begin{equation}\label{eqn:multi}
	\rho^2 \eta_1\eta_2= \frac{1}{2}(Y_1-Y_2),
\end{equation}
where
\begin{equation}
	\begin{split}
	Y_1&=\rho^{2(j_1+j_2)+ \frac{k_1+k_2}{\beta}}(\log \rho)^{m_1+m_2}   \cos(l_1+l_2)\theta \\	
	Y_2&=\rho^{2(j_1+j_2)+ \frac{k_1+k_2}{\beta}}(\log \rho)^{m_1+m_2}   \cos(l_1-l_2)\theta.
	\end{split}
	\label{eqn:Y}
\end{equation}

By the definition of $\mathcal T_{rhs}$, it suffices to show that (a) $Y_1$ and $Y_2$ are in $\mathcal T_{\log}$, but (b) neither $Y_1$ or $Y_2$ is in 
\begin{equation}\label{eqn:badset}
	\set{1, \rho^{\frac{k}{\beta}}(\log \rho)^{k-1} \cos k\theta|\, k\in \mathbb N}.	
\end{equation}
The proof of (a) is trivial and is in fact the same as the proof of the claim that $Span(\mathcal T_{\log})$ is multiplicatively closed.

For (b), we first notice that neither $Y_1$ or $Y_2$ can be constant. Otherwise, we must have
\begin{equation*}
	2(j_1+j_2)+\frac{k_1+k_2}{\beta}=m_1+m_2=0,
\end{equation*}
and hence $j_2=k_2=m_2=0$, which is a contradiction to $\eta_2\in \mathcal T_{rhs}$.

If there is $k\in \mathbb N$ such that $Y_1 = \rho^{\frac{k}{\beta}}(\log \rho)^{k-1} \cos k\theta$, then
\begin{align}
	\label{a1}	2(j_1+j_2) \beta + (k_1+k_2) &= k \\
	\label{a2} 	m_1+m_2 &= k-1 \\
	\label{a3}	l_1+l_2 &= k.
\end{align}
By $k_1\geq l_1$, $k_2\geq l_2$ and $j_1,j_2\geq 0$, we know
\begin{equation*}
	j_1=j_2=0,\quad k_1=l_1, \quad k_2=l_2.
\end{equation*}

If $k_1=0$, then $k_1=m_1=l_1=j_1=0$, i.e. $\eta_1=1$, in which case the lemma is trivial and there is nothing to prove. If $k_2=0$, then $k_2=m_2=l_2=j_2=0$ and $\eta_2=\rho^{-2}$, which is a contradiction to the assumption that $\eta_2\in \mathcal T_{rhs}$.

If both $k_1$ and $k_2$ are positive, then $m_1\leq k_1-1$ and $m_2\leq k_2-1$, which contradicts \eqref{a2}. In summary, we have proved that $Y_1$ is not in \eqref{eqn:badset}.

If there is $k\in \mathbb N$ such that $Y_2 = \rho^{\frac{k}{\beta}}(\log \rho)^{k-1} \cos k\theta$, then
\begin{align}
	\label{b1}	2(j_1+j_2) \beta + (k_1+k_2) &= k \\
	\label{b2} 	m_1+m_2 &= k-1 \\
	\label{b3}	\abs{l_1-l_2} &= k.
\end{align}
A similar discussion yields a contradiction. Hence, the proof of (1) is done.

\vskip 2mm
(2) 
Assume 
\begin{equation*}
	\eta_1=\rho^{2j_1+\frac{k_1}{\beta}}(\log \rho)^{m_1} \cos l_1 \theta \quad \text{and} \quad \eta_2=\rho^{2j_2+\frac{k_2}{\beta}}(\log \rho)^{m_2} \cos l_2 \theta
\end{equation*}
and let $Y_1$ and $Y_2$ be defined as in \eqref{eqn:Y}. Since $\eta_1\cdot \eta_2=\frac{1}{2}(Y_1-Y_2)$, it suffices to prove that (as before) both $Y_1$ and $Y_2$ are in $\mathcal T_{\log}$, but neither is in \eqref{eqn:badset}. Again, the first assertion is trivial. 

By our assumption that neither of $\eta_1$ and $\eta_2$ is constant function, we have
\begin{equation}\label{eqn:nonzero}
	j_1 \beta + k_1>0 \quad \text{and} \quad j_2 \beta +k_2>0.
\end{equation}
A consequence of this is that $Y_1$ and $Y_2$ are not constant function.

It remains to exclude the possibility that $Y_1$ (or $Y_2$) is $\rho^{\frac{k}{\beta}}(\log \rho)^{k-1} \cos k\theta$ for some $k\in \mathbb N$. In that case, \eqref{a1} and \eqref{a2} hold as before and imply that
\begin{equation}\label{eqn:equal}
	2(j_1+j_2)\beta +(k_1+k_2) = m_1+m_2+1.
\end{equation}
This is a contradiction to $m_1\leq \max\set{0,k_1-1}$ and $m_2\leq \max \set{0,k_2-1}$, unless (at least) one of $k_1$ and $k_2$ is zero. If $k_1=m_1=0$, then \eqref{eqn:nonzero} implies that $j_1>0$, and hence \eqref{eqn:equal} implies
\begin{equation*}
	k_2< m_2+1,
\end{equation*}
which is a contradiction to the definition of $\mathcal T_{\log}$.
\end{proof}

\subsection{Finite expansion}\label{subsec:finite}
Throughout this section, we fix $\xi$ and take $\varphi$ and $\varphi_{\chi}$ as functions of $\rho$ and $\theta$ alone. We write $B_r$ for the set of $(\rho,\theta)$ with $\rho\in (0,r)$ and $\theta\in S^1$.

\begin{defn}
	\label{def:O} A function $u$ defined in $B_{1/2}$ is said to be in $\tilde{O}(q)$ for some $q\in \Real$ if and only if there are constants $C(k_1,k_2)$ for all $k_1,k_2=0,1,2,\cdots$ such that
		\begin{equation*}
			\abs{(\rho \partial_\rho)^{k_1} \partial_\theta^{k_2} u}\leq C(k_1,k_2)\rho^q \quad \mbox{on} \quad B_{1/2}.
		\end{equation*}
\end{defn}
\begin{rem}
	Sometimes, we abuse the notation by using $\tilde{O}(q)$ to denote a function in it.
\end{rem}
Theorem \ref{thm:interior} implies that $\varphi$ and $\varphi_\chi$ are in $\tilde{O}(0)$. In fact, we have 
\begin{lem}\label{lem:start}
	Suppose that $u$ is in both $\tilde{O}(0)$ and Donaldson's space $C^{2,\alpha}_\beta$. Then there is some $q>1$ such that
	\begin{equation*}
		u=c+ \tilde{O}(q)
	\end{equation*}
	for some constant $c$.
\end{lem}

\begin{proof}
	Being in $C^{2,\alpha}_\beta$ implies that
	\begin{equation*}
		\abs{\pfrac{u}{\rho}}+ \abs{\frac{1}{\rho}\pfrac{u}{\theta}} \leq C \rho^{\alpha}
	\end{equation*}
	and that $u(0,\theta)$ is a constant independent of $\theta$.
	By the Newton-Lebnitz formula
	\begin{equation*}
		u(\rho,\theta)= u(0,\theta) + \int_0^\rho \pfrac{u}{\rho}(t,\theta) dt,
	\end{equation*}
	we have
	\begin{equation}\label{eqn:uonealpha}
		\abs{u(\rho,\theta)-u(0,\theta)}\leq C \rho^{1+\alpha}.
	\end{equation}
	For each fixed $\rho_0$ and $\theta_0$, set $v(\rho,\theta)= u(\rho \rho_0,\theta)- u(0,\theta)$. If $v$ is regarded as a function of $(\rho,\theta)$ defined on $Q:=\set{(\rho,\theta)|\, \frac{1}{2}<\rho<2, \theta_0-0.1<\theta<\theta_0+0.1}$, we have
	\begin{equation*}
		\norm{v}_{C^k(Q)}\leq C(k)\quad \mbox{for any} 	\quad k
	\end{equation*}
	because $u$ is in $\tilde{O}(0)$. On the other hand, \eqref{eqn:uonealpha} gives
	\begin{equation*}
		\norm{v}_{C^0(Q)}\leq C \rho_0^{1+\alpha}.
	\end{equation*}
	The usual interpolation yields that
	\begin{equation}\label{eqn:goodv}
		\norm{v}_{C^k(Q)}\leq C'(k) \rho_0^{q}
	\end{equation}
	for any $1<q<1+\alpha$ and another sequence of constants $C'(k)$ (depending on $q$). The lemma follows if \eqref{eqn:goodv} is translated into inequalities of $u$.
\end{proof}

\begin{defn}
	A function $u$ is said to have an expansion up to order $q$ if and only if there is $\eta \in Span(\mathcal T_{\log})$ such that
	\begin{equation*}
		u=\eta + \tilde{O}(q).
	\end{equation*}
	Similarly, a function $u$ is said to have an (rhs)-expansion up to order $q$ if and only if there is $\eta \in Span(\mathcal T_{rhs})$ such that the above holds.
\end{defn}

Obviously, if $2j+\frac{k}{\beta}>q$ or $2j+\frac{k}{\beta}=q$ with $m=0$, then
\begin{equation*}
	\rho^{2j+\frac{k}{\beta}}(\log \rho)^m \cos l\theta \in \tilde{O}(q).
\end{equation*} 
Therefore, $\eta$ in the above definition of expansion is only unique up to terms like these. Similar observation applies to the (rhs)-expansion.

Lemma \ref{lem:start} and the fact that $\varphi$ and $\varphi_\chi$ are in $C^{2,\alpha}_\beta$ imply that $\varphi$ and $\varphi_\chi$ have expansion up to order $q$ for some $q>1$. The proof of the main theorem is a bootstrapping argument starting from this.

For the use of the next section, we need the following three lemmas, which explains the reason why we want $Span(\mathcal T_{log})$ to be multiplicatively closed.

\begin{lem}
	\label{lem:product}
	Suppose that $u_1$ and $u_2$ have expansions up to order $q$. Then so does $u_1\cdot u_2$.
\end{lem}

\begin{lem}
	\label{lem:OO} If $u_i$ is in $\tilde{O}(q_i)$ for $i=1,2$, then $u_1\cdot u_2$ is in $\tilde{O}(q_1+q_2)$.  
\end{lem}

\begin{lem}\label{lem:smooth}
	Suppose $F(x_1,\cdots,x_N)$ is a smooth function of $N$ variables. Assume that $u_1,\cdots,u_N$ have expansions up to order $q$ for some $q\geq 0$. Then $F(u_1,\cdots,u_N)$ has an expansion up to order $q$.
\end{lem}

The proof of Lemma \ref{lem:OO} is trivial and the proof of Lemma \ref{lem:product} is the combination of the facts that $Span(\mathcal T_{\log})$ is multiplicatively closed, that any function in $Span(\mathcal T_{\log})$ is in $\tilde{O}(0)$ and that for $v_1(v_2)$ in $\tilde{O}(q_1)(\tilde{O}(q_2))$ respectively, we have $v_1\cdot v_2 \in \tilde{O}(q_1+q_2)$. All the above mentioned facts can be verified directly by Definition \ref{def:O}.

Lemma \ref{lem:smooth} is a simple generalization of the Lemma 6.8 in \cite{yin2016}, where the case $N=1$ is proved. For general $N$, it suffices to replace the Taylor expansion formula of one variable by the Taylor expansion of multiple variables. We refer the readers to \cite{yin2016} for details of the proof.

\subsection{The proof of Theorem \ref{thm:main}}\label{subsec:proof}
The proof of the main theorem relies on the following two lemmas.
\begin{lem}
	\label{lem:first} If $\varphi$ and all of $\varphi_{\chi}$ have expansion up to order $q$ for some $q>1$, then the right hand side of ($B_\chi$) has a (rhs)-expansion up to order $q-1$.
\end{lem}
\begin{lem}
	\label{lem:second}
	If $v$ has a (rhs)-expansion up to order $q$ and $u$ is a bounded solution to
	\begin{equation*}
		\tilde{\triangle} u =v,
	\end{equation*}
	then $u$ has an expansion up to order $q'$ for any $q'<q+2$.
\end{lem}

Before the proof of these two lemmas, we show how they imply Theorem \ref{thm:main}.
\begin{proof}[Proof of Theorem \ref{thm:main} (assuming Lemma \ref{lem:first} and \ref{lem:second})]
	We have as our starting point that $\varphi$ and $\varphi_{\chi}$ are in Donaldson's $C^{2,\alpha}_\beta$ space. Theorem \ref{thm:interior} and Lemma \ref{lem:start} imply that $\varphi$ and $\varphi_{\chi}$ have expansions up to some order $q$ with $q>1$.
	The rest of the proof is easy bootstrapping argument using Lemma \ref{lem:first} and \ref{lem:second}.
\end{proof}

The proof of Lemma \ref{lem:first} depends heavily on the structure of the complex Monge-Amp\`ere equation.
\begin{proof}[Proof of Lemma \ref{lem:first}]
	The right hand side of ($B_{\chi}$) consists of three terms:
	\begin{enumerate}[(a)]
		\item 
			\begin{equation*}
				F(\varphi_{\chi_2,\ldots,\varphi_{\chi_{l+2}}}) \# (P \varphi_{\chi_a}\cdot \bar{P} \varphi_{\chi_b});	
			\end{equation*}
		\item 
			\begin{equation*}
				H(h,\cdots,h_{\chi_l},\varphi,\cdots,\varphi_{\chi_l});
			\end{equation*}
		\item
			\begin{equation*}
		(\det(\varphi_{\xi\bar{\xi}})(0,\xi)- \det(\varphi_{\xi\bar{\xi}})(\rho,\theta,\xi)) \tilde{\triangle} \varphi_\chi.
			\end{equation*}
\end{enumerate}
We discuss (b) first. By our assumption in Theorem \ref{thm:main}, $h$ and its tangent derivatives are smooth in $z_1$. Hence, they have expansion up to any order. One can either check this directly by using the definition, or use Lemma \ref{lem:smooth} and notice that $z_1$ (as a function of $(\rho,\theta)$) has expansion up to any order. By the assumptions about $\varphi$ in Lemma \ref{lem:first}, Lemma \ref{lem:smooth} again implies that $H$ has expansion up to order $q$. Since $\mathcal T_{rhs}\supset \mathcal T_{\log}$, the term (b) has the required expansion.

For (c), we notice that there are $\eta_1,\eta_2\in Span(\mathcal T_{\log})$ such that
\begin{equation*}
	\det(\varphi_{\xi\bar{\xi}})(0,\xi)- \det(\varphi_{\xi\bar{\xi}})(\rho,\theta,\xi)= \eta_1 + \tilde{O}(q)
\end{equation*}
and
\begin{equation*}
	\tilde{\triangle} \varphi_\chi = \rho^{-2} \eta_2 + \tilde{O}(q-2).
\end{equation*}
Here, the existence of $\eta_1$ follows from Lemma \ref{lem:smooth} and the assumptions of the lemma. For $\eta_2$, we first find $\tilde{\eta}_2$ satisfying $\varphi_\chi = \tilde{\eta}_2 + \tilde{O}(q)$ and then check that $\tilde{\triangle}$ maps $\tilde{O}(q)$ to $\tilde{O}(q-2)$ and that for any $v\in Span(\mathcal T_{\log})$, $\tilde{\triangle}v$ is in $Span(\rho^{-2} \mathcal T_{\log})$ (see \eqref{eqn:log}).

We claim that
\begin{equation}\label{eqn:eta12}
	\eta_1,\eta_2 \in \tilde{O}(\tilde{q})\quad \mbox{for some } \quad \tilde{q}>1.
\end{equation}
On one hand, every function (hence $\eta_1$ and $\eta_2$) in $Span(\mathcal T_{\log})$ is in $\tilde{O}(0)$. On the other hand, it is easy to see that both $\eta_1$ and $\eta_2$ have no constant term. Note that all the rest of the functions in $\mathcal T_{\log}$ are bounded by $C\rho^{\tilde{q}}$ for some $\tilde{q}>1$ (depending on $\beta$). 

With \eqref{eqn:eta12} in mind, we compute (c) as
\begin{eqnarray*}
	&& (\eta_1+\tilde{O}(q))(\rho^{-2}\eta_2+\tilde{O}(q-2)) \\
	&=& \rho^{-2} \eta_1\cdot \eta_2 + \eta_1\cdot \tilde{O}(q-2) + \rho^{-2}\eta_2 \cdot \tilde{O}(q) + \tilde{O}(q)\cdot \tilde{O}(q-2).
\end{eqnarray*}
Given \eqref{eqn:eta12} and Lemma \ref{lem:OO}, the sum of the last three terms above is in $\tilde{O}(q-1)$. The second part of Lemma \ref{lem:logrhs} implies that $\rho^{-2} \eta_1\cdot \eta_2$ is in $Span(\mathcal T_{rhs})$ and hence finishes the proof for (c).

For the proof for (a), we claim that there exists $\eta_4\in Span(\mathcal T_{rhs})$ such that
\begin{equation}\label{eqn:P}
	P\varphi_{\chi_a}\cdot \bar{P}\varphi_{\chi_b} = \eta_4 + \tilde{O}(q-1).
\end{equation}
Before we prove the claim, we see how it implies that (a) has the required expansion. By Lemma \ref{lem:smooth}, we know that $F$ has expansion up to order $q$ so that
\begin{equation*}
	F \# (P\varphi_{\chi_a}\cdot \bar{P}\varphi_{\chi_b}) = (\eta_3+\tilde{O}(q))(\eta_4+\tilde{O}(q-1))
\end{equation*}
for some $\eta_3\in Span(\mathcal T_{\log})$.
By the first part of Lemma \ref{lem:logrhs}, $\eta_3\cdot \eta_4$ is in $Span(\mathcal T_{rhs})$. The remaining terms, $\eta_4\cdot \tilde{O}(q)$, $\eta_3\cdot \tilde{O}(q-1)$ and $\tilde{O}(q-1)\cdot \tilde{O}(q)$ are obviously in $\tilde{O}(q-1)$.

The rest of the proof is devoted to the proof of the claim. By our assumption, we may assume that
\begin{equation*}
	\varphi_{\chi_a}= \eta_a + \tilde{O}(q) \quad \mbox{and} \quad \varphi_{\chi_b}=\eta_b +\tilde{O}(q)
\end{equation*}
for $\eta_a$ and $\eta_b$ in $Span(\mathcal T_{\log})$. By Lemma \ref{lem:pair}, the right hand side of \eqref{eqn:P} is a quadratic polynomial of
\begin{equation*}
	\partial_\rho \eta_a + \tilde{O}(q-1),\quad \frac{1}{\rho}\partial_\theta \eta_a + \tilde{O}(q-1)
\end{equation*}
and
\begin{equation*}
	\partial_\rho \eta_b + \tilde{O}(q-1),\quad  \frac{1}{\rho}\partial_\theta \eta_b + \tilde{O}(q-1).
\end{equation*}
By the definition of $\mathcal T_{\log}$, $\partial_\rho \eta_a$, $\partial_\rho \eta_b$, $\frac{1}{\rho}\partial_\theta \eta_a$ and $\frac{1}{\rho}\partial_\theta \eta_b$ are in $\tilde{O}(0)$ because among all terms in $\mathcal T_{\log}$, except the constant term which is killed by the derivative, the lowest order term decays faster than $\rho$. This implies that the terms like $\partial_\rho \eta_a \cdot \tilde{O}(q-1)$ is in $\tilde{O}(q-1)$.

It remains to show $\partial_\rho \eta_a \cdot \partial_\rho \eta_b$ is in $Span(\mathcal T_{rhs})$. Similar argument works for $\partial_\rho \eta_a \cdot \frac{1}{\rho}\partial_\theta \eta_b$, $\frac{1}{\rho}\partial_\theta \eta_b\cdot \partial_\rho \eta_a$ and $\frac{1}{\rho^2} \partial_\theta \eta_a \cdot \partial_\theta \eta_b$. 

This is a consequence of the second part of Lemma \ref{lem:logrhs}. In fact, if we set $\tilde{\eta}_a= (\rho\partial_\rho) \eta_a$ and $\tilde{\eta}_b= (\rho\partial_\rho) \eta_b$, then $\partial_\rho \eta_a \cdot \partial_\rho \eta_b = \rho^{-2} \tilde{\eta}_a \cdot \tilde{\eta}_b$ and we can check that $\tilde{\eta}_a$ and $\tilde{\eta}_b$ are non-constant function in $Span(\mathcal T_{\log})$.
\end{proof}

Next, we move to the proof of Lemma \ref{lem:second}.
\begin{proof} [Proof of Lemma \ref{lem:second}]
	Note that we apply the following argument to every equation ($B_\chi$) simultaneously. By the assumption, the right hand side is given by $\eta+\tilde{O}(q)$ for some $\eta\in Span(\mathcal T_{rhs})$. Lemma \ref{lem:rhs} implies that there exists $\eta'$ in $Span(\mathcal T_{log})$ such that
	\begin{equation*}
		\tilde{\triangle} \eta' =\eta.
	\end{equation*}
	Pretending the constant on the left hand side of ($B_\chi$) is $1$, we have
	\begin{equation}\label{eqn:almostdone}
		\tilde{\triangle} (\varphi_{\chi}-\eta') = \tilde{O}(q).
	\end{equation}
	To finish the proof of Lemma \ref{lem:second}, we need
	\begin{lem}\label{lem:harmonic}
	(1) (Lemma 6.7 in \cite{yin2016}) Any bounded ($\tilde{\triangle}$) harmonic function defined on $\set{(\rho,\theta)|\quad \rho\in (0,1], \theta\in S^1}$ has expansion up to any order.

	(2) (Lemma 6.10 in \cite{yin2016})For each $v\in \tilde{O}(q)$ and $q'<q+2$, there is $u$ in $\tilde{O}(q')$ with
		\begin{equation*}
			\tilde{\triangle} u =v.
		\end{equation*}
	\end{lem}
	Since this lemma appeared in exactly the same form and is proved in detail in \cite{yin2016}, we refer the readers to that paper for the proof.

	With the above lemma and \eqref{eqn:almostdone}, there exists $f\in \tilde{O}(q')$ for any $q'<q+2$ satisfying
	\begin{equation*}
		\tilde{\triangle} (\varphi_\chi-\eta')= \tilde{\triangle}(f),
	\end{equation*}
	which implies that $\varphi_\chi-\eta'-f$ is a bounded harmonic function and hence has expansion up to any order by (1) of Lemma \ref{lem:harmonic}.
\end{proof}

\subsection{The proof of Corollary \ref{cor:main}} \label{subsec:cor}
By \eqref{eqn:cor} and \eqref{eqn:mainh}, we know $\varphi$ satisfies \eqref{eqn:main}. To apply Theorem \ref{thm:main}, we need to check (S1)-(S3). 

Since $\psi_0$ is the potential function of the smooth K\"ahler form $\omega_0$ in the holomorphic coordinates $\set{z_i}$, we know that $\psi_0$ is smooth and that $(\psi_0)_{i\bar{j}}$ is smooth and positive definite. Moreover, being the metric of the line bundle $L$, $\tilde{h}$ is also smooth and positive. Hence (S1) follows from \eqref{eqn:mainh}.

It follows from the definition of $C^{2,\alpha}_\beta$ (see (D1)-(D4) in Section \ref{subsec:estimate}) and $\beta\in (0,1)$ that smooth functions in holomorphic coordinates are in $C^{2,\alpha}_\beta$ for some $\alpha>0$ depending on $\beta$. So $\psi_0, \tilde{h} \in C^{2,\alpha}_\beta$. Moreover, one can check directly that $\abs{z_1}^{2\beta}$ is also in $C^{2,\alpha}_\beta$. The best way to see this is to notice that
\begin{equation*}
	\abs{z_1}^{2\beta}= \beta^2 \rho^2
\end{equation*}
and to use the equivalent definition (P1)-(P3).  Therefore, $\abs{s}_{h_0}^{2\beta}= \tilde{h}^\beta \abs{z_1}^{2\beta}$ is in $C^{2,\alpha}_\beta$. Finally, since $\psi$ is in $C^{2,\alpha}_\beta$ as assumed in the definition of conical K\"ahler-Einstein metric, so is $\varphi$, which confirms (S2) (for some $\alpha>0$).

(S3) follows from the definition of conical K\"ahler-Einstein metric. 

Now, we can apply Theorem \ref{thm:main} to see that $\varphi$ has the expansion up to any order. To see that this is also true for $\psi$, it suffices to check that 
\begin{equation*}
	\varphi-\psi = \psi_0 + \delta \abs{s}_{h_0}^{2\beta} \quad (\text{by} \quad \text{\eqref{eqn:varphi}})
\end{equation*}
has expansion up to any order. 
Since $\psi_0$ and $\tilde{h}$ are smooth in holomorphic coordinates and
\begin{equation*}
	\abs{s}_{h_0}^{2\beta} = \tilde{h}^\beta \abs{z_1}^{2\beta} = \tilde{h}^{\beta} \beta^2 \rho^2,
\end{equation*}
it remains to check all smooth functions in holomorphic coordinates have expansions up to any order, which is a consequence of Lemma \ref{lem:smooth} and the fact that $z_1$ has expansion up to any order.

%

\appendix

\section{Estimate of some elliptic system}
In this appendix, we prove some estimates for the following elliptic system defined on $B_1\subset \mathbb C^n$,
\begin{equation}
	\triangle_g g_{i\bar{j}} - g^{k\bar{m}} g^{n\bar{l}} \pfrac{g_{i\bar{m}}}{z_n} \pfrac{g_{k\bar{j}}}{z_{\bar{l}}} = \lambda\rho_0^2 g_{i\bar{j}} + h_{\tilde{V},i\bar{j}} .
	\label{eqn:gia}
\end{equation}
When $h=0$, this is the equation satisfied by the metric tensor of K\"ahler-Einstein metric. We always assume 
\begin{equation}
	\abs{\lambda \rho_0^2} + \norm{h_{\tilde{V},i\bar{j}}}_{C^0(B_1)}\leq \Lambda.
	\label{eqn:Lambda}
\end{equation}

The methods used here are from the book of Giaquinta \cite{giaquinta1983multiple} and are by now classical. In contrast to the theorems in \cite{giaquinta1983multiple}, we prove effective estimates instead of just regularity statements. 

\begin{rem*}
	Note that in this appendix, the real dimension of the domain is $2n$, instead of $n$.
\end{rem*}

We first bound the $L^2$ norm of the gradient of $g_{i\bar{j}}$.
\begin{lem}
	\label{lem:energy}
	Suppose that $g_{i\bar{j}}$ are some smooth complex-valued functions defined on $B_1\subset \mathbb C^n$ solving \eqref{eqn:gia} whose coefficients satisfy \eqref{eqn:Lambda}. If for some $\sigma>0$, we have
	\begin{equation*}
		\sigma^{-1} g_{i\bar{j}}\leq \delta_{ij}\leq \sigma g_{i\bar{j}} \quad \mbox{on} \quad B_1,
	\end{equation*}
	then 
	\begin{equation*}
		\norm{\nabla g_{i\bar{j}}}_{L^2(B_{3/4})}\leq C (\sigma,\Lambda,\norm{g_{i\bar{j}}}_{C^\alpha(B_1)}).	
	\end{equation*}
\end{lem}

\begin{proof}
	For any point $x_0$ in $B_{3/4}$ and $R>0$ to be determined in the proof, let $\eta$ be some smooth cut-off function supported in $B_R(x_0)$ with $\eta\equiv 1$ in $B_{R/2}(x_0)$ and $\abs{\nabla \eta}\leq CR^{-1}$. Multiplying both sides of \eqref{eqn:gia} by $(g_{i\bar{j}}-g_{i\bar{j}}(x_0)) \eta^2$ and freezing the coefficients of the leading term in \eqref{eqn:gia} gives
	\begin{eqnarray}\label{eqn:gia1}
	0&=& g^{k\bar{l}}(x_0) \partial_k\bar{\partial}_l g_{i\bar{j}} (g_{i\bar{j}}-g_{i\bar{j}}(x_0)) \eta^2 \\ \nonumber
	&& - (g^{k\bar{l}}(x_0)-g^{k\bar{l}}) \partial_k \bar{\partial}_l g_{i\bar{j}} (g_{i\bar{j}}-g_{i\bar{j}}(x_0)) \eta^2 \\ \nonumber
	&& - (g\cdot g\cdot D g\cdot D g +\lambda \rho_0^2 g +h) \cdot (g_{i\bar{j}}-g_{i\bar{j}}(x_0)) \eta^2 .
\end{eqnarray}
Note that we have omitted subscripts in the above computation when they are not essential to the proof.

By the H\"older continuity of $g_{i\bar{j}}$, we have
\begin{equation}\label{eqn:hlder}
	\abs{g_{i\bar{j}}-g_{i\bar{j}}(x_0)} \leq C R^{\alpha}.
\end{equation}
Integration by parts of \eqref{eqn:gia1}, \eqref{eqn:hlder} and Young's inequality imply that
\begin{eqnarray*}
	\int \abs{D g}^2 \eta^2 &\leq& C \int R^\alpha \abs{D g} \eta \abs{\nabla \eta} + R^\alpha \abs{D g}^2 \eta^2 \\
	&& + C\int R^{2\alpha} \abs{D g} \eta \abs{\nabla \eta} + R^\alpha \abs{D g}^2 \eta^2 + C R^{2n+\alpha} \\
	&\leq& (\frac{1}{2}+ CR^\alpha) \int \abs{D g}^2 \eta^2 + CR^{2\alpha} \int \abs{\nabla \eta}^2 + C R^{2n+\alpha}.
\end{eqnarray*}
Now we can choose $R$ so small (depending only on $\sigma, \alpha$ and the H\"older norm of $g_{i\bar{j}}$) that the first term in the right hand side is absorbed by the left hand side to give
\begin{equation}\label{eqn:bettercalpha}
	\int_{B_{R/2}(x_0)} \abs{D g}^2 dx \leq CR^{2n-2+2\alpha}\leq C.
\end{equation}
The lemma then follows from the above inequality by covering $B_{3/4}$ by balls of radius $R$.
\end{proof}

Next, we prove $C^\gamma$ estimate of $g_{i\bar{j}}$ for any $\alpha<\gamma<1$.
\begin{lem}\label{lem:gia2}
	For $g_{i\bar{j}}$ as in Lemma \ref{lem:energy} and any $\alpha<\gamma<1$, we have
	\begin{equation*}
		\norm{g_{i\bar{j}}}_{C^{\gamma}(B_{3/4})} \leq C(\sigma,\Lambda,\gamma,\norm{g_{i\bar{j}}}_{C^\alpha(B_1)}).
	\end{equation*}
\end{lem}
\begin{proof}
	For any $x_0\in B_{3/4}$ and $R\leq 1/4$, let $v_{i\bar{j}}$ be the solution of
	\begin{equation*}
		\left\{
			\begin{array}[]{ll}
				g^{k\bar{l}}(x_0) \partial_k \bar{\partial}_l v_{i\bar{j}} =0 & \quad \mbox{on} \quad B_R(x_0) \\
				v_{i\bar{j}}= g_{i\bar{j}} &\quad \mbox{on} \quad \partial B_R(x_0).
			\end{array}
			\right.
	\end{equation*}
	By Theorem 2.1 on page 78 of \cite{giaquinta1983multiple} (applied to $Dv_{i\bar{j}}$), there is a constant depending only on $\sigma$ such that for $0<\rho<R$,
	\begin{equation}\label{eqn:Dvv}
		\int_{B_{\rho}(x_0)} \abs{Dv_{i\bar{j}}}^2 \leq C(\frac{\rho}{R})^{2n} \int_{B_R(x_0)} \abs{Dv_{i\bar{j}}}^2.
	\end{equation}
	Setting $w_{i\bar{j}}=g_{i\bar{j}}-v_{i\bar{j}}$, we get (using \eqref{eqn:Dvv})
	\begin{eqnarray}\label{eqn:young2}
		&& \int_{B_{\rho}(x_0)} \abs{Dg_{i\bar{j}}}^2 \\ \nonumber
		&\leq & \int_{B_{\rho}(x_0)} \abs{Dv_{i\bar{j}}}^2+ \int_{B_{\rho}(x_0)} \abs{Dw_{i\bar{j}}}^2\\\nonumber
		&\leq& C(\frac{\rho}{R})^{2n} \int_{B_{R}(x_0)} \abs{Dv_{i\bar{j}}}^2 + C \int_{B_{\rho}(x_0)} \abs{Dw_{i\bar{j}}}^2 \\\nonumber
		&\leq& C(\frac{\rho}{R})^{2n} \int_{B_{R}(x_0)} \abs{Dg_{i\bar{j}}}^2 + C \int_{B_{R}(x_0)} \abs{Dw_{i\bar{j}}}^2.
	\end{eqnarray}
	Using the equation satisfied by $v_{i\bar{j}}$, we may rewrite \eqref{eqn:gia} as follows
	\begin{equation*}
		g^{k\bar{l}}(x_0) \partial_k\bar{\partial}_l (g_{i\bar{j}}-v_{i\bar{j}})= -(g^{k\bar{l}}-g^{k\bar{l}}(x_0))\partial_k\bar{\partial}_l g_{i\bar{j}} + g\cdot g\cdot D g \cdot D g + \lambda \rho_0^2 g + h.
	\end{equation*}
	Since $w_{i\bar{j}}$ vanishes on $\partial B_R(x_0)$, we can use it as the test function of the above equation to obtain
	\begin{eqnarray*}
		\int_{B_R(x_0)} \abs{Dw_{i\bar{j}}}^2 &\leq& C \int_{B_R(x_0)} \partial_k \bar{\partial}_l g_{i\bar{j}} \left( w (g^{k\bar{l}}-g^{k\bar{l}}(x_0)) \right) + \abs{w}\abs{D g}^2 + \abs{w} \\
		&\leq& C \int_{B_R(x_0)} \abs{D g}\abs{\nabla w} \abs{g^{k\bar{l}}-g^{k\bar{l}}(x_0)} + \abs{w}\abs{D g}^2 + \abs{w}.
	\end{eqnarray*}
	Using Young's inequality, we get
	\begin{equation}\label{eqn:young1}
		\int_{B_R(x_0)} \abs{Dw_{i\bar{j}}}^2 \leq C \int_{B_R(x_0)} \abs{D g}^2 (\abs{g^{k\bar{l}}-g^{k\bar{l}}(x_0)}^2 + \abs{w}) + \abs{w}.
	\end{equation}
	The maximum principle implies that $\osc_{B_R(x_0)} v_{i\bar{j}} \leq osc_{B_R(x_0)} g_{i\bar{j}}$, which implies that
	\begin{equation}\label{eqn:young3}
		\norm{w}_{C^0(B_R(x_0))}\leq \osc_{B_R(x_0)} v_{i\bar{j}}+ \osc_{B_R(x_0)} g_{i\bar{j}}\leq  C R^\alpha.
	\end{equation}
	Putting \eqref{eqn:young2}, \eqref{eqn:young1} and \eqref{eqn:young3} together yields the following decay estimate
	\begin{equation*}
		\int_{B_\rho(x_0)} \abs{Dg}^2 \leq C\left( (\frac{\rho}{R})^{2n} + R^\alpha \right) \int_{B_R(x_0)} \abs{Dg}^2 + C R^{2n+\alpha}.
	\end{equation*}
	Dividing both sides of the above equation by $\rho^{2n-2}$ and setting $\rho/R=\tau$ give
	\begin{equation*}
		\rho^{2-2n}\int_{B_\rho(x_0)} \abs{Dg}^2 \leq \left[ C (1+ \tau^{-2n} R^{\alpha}) \right] \tau^2 R^{2-2n}\int_{B_R(x_0)} \abs{Dg}^2 + C \tau^{2-2n} R^2.
	\end{equation*}
	By picking $\tau$ small so that $2C \tau^{(2-2\gamma)}=1$ and then $R_1$ small so that $1+\tau^{-n}R^\alpha<2$ for all $R<R_1$, we have 
	\begin{equation*}
		\rho^{2-2n}\int_{B_\rho(x_0)} \abs{Dg}^2 \leq \tau^{2\gamma} R^{2-2n}\int_{B_R(x_0)} \abs{Dg}^2 + C(\gamma) R^2
	\end{equation*}
	for $R<R_1$. From here, a routine iteration shows that
	\begin{equation*}
		\rho^{2-2n}\int_{B_\rho(x_0)} \abs{Dg}^2 \leq C \rho^{2\gamma}.
	\end{equation*}
	Hence, by the H\"older inequality and the equivalence between the Companato space and the H\"older space, we have
	\begin{equation*}
		\norm{g_{i\bar{j}}}_{C^\gamma(B_{3/4})}\leq C(\sigma,\Lambda,\gamma, \norm{g_{i\bar{j}}}_{C^\alpha(B_1)}).
	\end{equation*}
\end{proof}

Finally, we prove the $C^{1,\alpha}$ estimate.
\begin{lem}\label{lem:gia}
	For $g_{i\bar{j}}$ as in Lemma \ref{lem:energy}, there exists some $\alpha'\in (0,1)$ such that 
	\begin{equation*}
		\norm{g_{i\bar{j}}}_{C^{1,\alpha'}(B_{1/2})} \leq C(\sigma,\Lambda,\norm{g_{i\bar{j}}}_{C^\alpha(B_1)}).
	\end{equation*}
\end{lem}
\begin{proof}
	For any $x_0\in B_{1/2}$ and $R\leq 1/4$, as in the proof of Lemma \ref{lem:gia2}, let $v_{i\bar{j}}$ be the solution of
	\begin{equation*}
		\left\{
			\begin{array}[]{ll}
				g^{k\bar{l}}(x_0) \partial_k \bar{\partial}_l v_{i\bar{j}} =0 & \quad \mbox{on} \quad B_R(x_0) \\
				v_{i\bar{j}}= g_{i\bar{j}} &\quad \mbox{on} \quad \partial B_R(x_0).
			\end{array}
			\right.
	\end{equation*}
	Again, applying Theorem 2.1 on page 78 of \cite{giaquinta1983multiple} to $Dv_{i\bar{j}}$, we get a constant depending only on $\sigma$ such that
	\begin{equation}\label{eqn:Dv}
		\int_{B_{\rho}(x_0)} \abs{Dv_{i\bar{j}}- (Dv_{i\bar{j}})_{x_0,\rho}}^2 \leq C(\frac{\rho}{R})^{2n+2} \int_{B_R(x_0)} \abs{Dv_{i\bar{j}}-(Dv_{i\bar{j}})_{x_0,R}}^2.
	\end{equation}
	Here $(u)_{x,r}$ means the average of $u$ in the ball $B_r(x)$.
	
	Next, we set $w_{i\bar{j}}=g_{i\bar{j}}-v_{i\bar{j}}$ on $B_R(x_0)$. Triangle inequalities and \eqref{eqn:Dv} imply that
	\begin{eqnarray}\label{eqn:decay2}
		&& \int_{B_{\rho}(x_0)} \abs{Dg_{i\bar{j}}- (Dg_{i\bar{j}})_{x_0,\rho}}^2 \\ \nonumber
		&\leq & \int_{B_{\rho}(x_0)} \abs{Dv_{i\bar{j}}- (Dv_{i\bar{j}})_{x_0,\rho}}^2+ \int_{B_{\rho}(x_0)} \abs{Dw_{i\bar{j}}- (Dw_{i\bar{j}})_{x_0,\rho}}^2\\ \nonumber
		&\leq& C(\frac{\rho}{R})^{2n+2} \int_{B_{R}(x_0)} \abs{Dv_{i\bar{j}}- (Dv_{i\bar{j}})_{x_0,R}}^2 + C \int_{B_{\rho}(x_0)} \abs{Dw_{i\bar{j}}}^2 \\ \nonumber
		&\leq& C(\frac{\rho}{R})^{2n+2} \int_{B_{R}(x_0)} \abs{Dg_{i\bar{j}}- (Dg_{i\bar{j}})_{x_0,R}}^2 + C \int_{B_{R}(x_0)} \abs{Dw_{i\bar{j}}}^2.
	\end{eqnarray}
	Using $w_{i\bar{j}}$ as the test function of \eqref{eqn:gia} as in the proof of Lemma \ref{lem:gia2} gives
	\begin{eqnarray*}
		\int_{B_R(x_0)} \abs{Dw_{i\bar{j}}}^2 &\leq& C \int_{B_R(x_0)} \partial_k \bar{\partial}_l g_{i\bar{j}} \left( w (g^{k\bar{l}}-g^{k\bar{l}}(x_0)) \right) + \abs{w}\abs{D g}^2 + \abs{w} \\
		&\leq& C \int_{B_R(x_0)} \abs{D g}\abs{\nabla w} \abs{g^{k\bar{l}}-g^{k\bar{l}}(x_0)} + \abs{w}\abs{D g}^2 + \abs{w}.
	\end{eqnarray*}
	The Young's inequality and the fact that $g_{i\bar{j}}$ lies in $C^\gamma(B_{3/4})$ imply that
	\begin{eqnarray*}
		\int_{B_R(x_0)} \abs{Dw}^2 &\leq& C R^{\gamma} \int_{B_R(x_0)} \abs{D g}^2    + CR^{2n+\gamma}.
	\end{eqnarray*}
	Notice that we have proved (see \eqref{eqn:bettercalpha})
	\begin{equation*}
		\int_{B_R(x_0)} \abs{Dg}^2 \leq C R^{2n-2+2\gamma}
	\end{equation*}
	in the proof of Lemma \ref{lem:gia2}. By Lemma \ref{lem:gia2}, we can choose and fix $\gamma$ so that
	\begin{equation}\label{eqn:decay1}
		\int_{B_R(x_0)}\abs{Dw}^2 \leq C R^{2n+\gamma'}
	\end{equation}
	for some $\gamma'>0$ (small).
	Combining \eqref{eqn:decay2} and \eqref{eqn:decay1}, we get
	\begin{equation*}
		\rho^{-2n}\int_{B_{\rho}(x_0)} \abs{Dg_{i\bar{j}}- (Dg_{i\bar{j}})_{x_0,\rho}}^2 \leq C(\frac{\rho}{R})^2 R^{-2n} \int_{B_R(x_0)}\abs{Dg_{i\bar{j}}- (Dg_{i\bar{j}})_{x_0,\rho}}^2 + C(R/\rho)^{2n} R^{\gamma'}.
	\end{equation*}
	As before, picking $\tau\in (0,1)$ with $C \tau^{2-\gamma'}=1$ and setting $\rho=\tau R$, we get
	\begin{equation*}
		\rho^{-2n}\int_{B_{\rho}(x_0)} \abs{Dg_{i\bar{j}}- (Dg_{i\bar{j}})_{x_0,\rho}}^2 \leq \tau^{\gamma'} R^{-2n} \int_{B_R(x_0)}\abs{Dg_{i\bar{j}}- (Dg_{i\bar{j}})_{x_0,\rho}}^2 + C(\tau) R^{\gamma'}.
	\end{equation*}
	Iteration again implies that 
	\begin{equation*}
		\int_{B_{\rho}(x_0)} \abs{Dg_{i\bar{j}}- (Dg_{i\bar{j}})_{x_0,\rho}}^2 \leq C \rho^{2n+\gamma'},
	\end{equation*}
	which concludes the proof of the lemma by Theorem 1.2 in Chapter III of \cite{giaquinta1983multiple}.
\end{proof}

\bibliography{foo}
\bibliographystyle{plain}

\end{document}